\documentclass[10pt]{article}

\usepackage[utf8]{inputenc}
\usepackage[T1]{fontenc}

\usepackage{epsf}
\usepackage{amsmath}

\allowdisplaybreaks

\usepackage[showframe=false]{geometry}
\usepackage{changepage}

\usepackage{epsfig}
\usepackage{amssymb}

\usepackage{amsthm}
\usepackage{setspace}
\usepackage{cite}
\usepackage{mcite}

\usepackage{algorithmic}  % This package provides an algorithmic environment fo describing algorithms.
\usepackage{algorithm}

\usepackage{shadow}
\usepackage{fancybox}
\usepackage{fancyhdr}

\usepackage{color}
\usepackage[usenames,dvipsnames,svgnames,table]{xcolor}
\newcommand{\bl}[1]{\textcolor{blue}{#1}}

\definecolor{mypurple}{rgb}{.4,.0,.5}

\usepackage[hyphens]{url}

\usepackage[colorlinks=true,
            linkcolor=black,
            urlcolor=black,
            citecolor=purple]{hyperref}

\usepackage{breakurl}
\def\y{{\bf y}}

\def\x{{\bf x}}

\def\x{{\mathbf x}}

\def\x{{\bf x}}
\def\y{{\bf y}}

\def\h{{\bf h}}

\def\be{\begin{equation}}
\def\ee{\end{equation}}
\def\ba{\left[\begin{array}}
\def\ea{\end{array}\right]}

\def\x{{\bf x}}
\def\y{{\bf y}}

\def\1{{\bf 1}}

\def\g{{\bf g}}
\def\0{{\bf 0}}

\def\erfinv{\mbox{erfinv}}
\def\erf{\mbox{erf}}
\def\erfc{\mbox{erfc}}

\def\cG{{\mathcal G}}

\def\mR{{\mathbb R}}

\def\mN{{\mathbb N}}
\def\mE{{\mathbb E}}

\def\lp{\left (}
\def\rp{\right )}

\def\y{{\bf y}}

\def\x{{\bf x}}

\def\x{{\mathbf x}}

\def\x{{\bf x}}
\def\y{{\bf y}}

\def\h{{\bf h}}

\def\be{\begin{equation}}
\def\ee{\end{equation}}
\def\ba{\left[\begin{array}}
\def\ea{\end{array}\right]}

\def\x{{\bf x}}
\def\y{{\bf y}}

\def\({\left (}
\def\){\right )}

\def\1{{\bf 1}}

\def\g{{\bf g}}
\def\0{{\bf 0}}

\def\cX{{\mathcal X}}

\def\cS{{\mathcal S}}
\def\cL{{\mathcal L}}
\def\cI{{\mathcal I}}

\def\cX{{\mathcal X}}

\def\erfinv{\mbox{erfinv}}

\usepackage{xcolor}
\usepackage{color}

\definecolor{darkgreen}{rgb}{0, 0.4,0}

\definecolor{purplebrown}{rgb}{0.5,0.1,0.6}

\definecolor{ultclupcol}{rgb}{0.1,0.5,0.5}

\definecolor{mytrycolor}{rgb}{0.5,0.7,0.2}

\definecolor{ultclupcola}{rgb}{.5,0,.5}

\definecolor{shadebrown}{rgb}{0.1,0.1,0.9}
\definecolor{lightblue}{rgb}{0.2,0,1}

\usepackage{fancybox}
\usepackage{graphicx}
\usepackage{epstopdf}
\usepackage{epsfig}
\usepackage{wrapfig}
\usepackage{subfigure}

\usepackage{xcolor}
\usepackage{tcolorbox}
\newtcbox{\xmybox}{on line,
arc=7pt,
before upper={\rule[-3pt]{0pt}{10pt}},boxrule=0pt,
boxsep=0pt,left=6pt,right=6pt,top=0pt,bottom=0pt,enhanced, coltext=blue, colback=white!10!yellow}

\newtcbox{\xmyboxa}{on line,
arc=7pt,
before upper={\rule[-3pt]{0pt}{10pt}},boxrule=0pt,
boxsep=0pt,left=6pt,right=6pt,top=0pt,bottom=0pt,enhanced, colback=white!10!yellow}

\newtcbox{\xmyboxb}{on line,
arc=7pt,
before upper={\rule[-3pt]{0pt}{10pt}},boxrule=1pt,colframe=darkgreen!100!blue,
boxsep=0pt,left=6pt,right=6pt,top=0pt,bottom=0pt,enhanced, colback=white!10!yellow}

\newtcbox{\xmyboxc}{on line,
arc=7pt,
before upper={\rule[-3pt]{0pt}{10pt}},boxrule=.7pt,colframe=blue!100!blue,
boxsep=0pt,left=6pt,right=6pt,top=0pt,bottom=0pt,enhanced, coltext=blue, colback=white!10!yellow}

\newtcbox{\xmytboxa}{on line,
arc=7pt,
before upper={\rule[-3pt]{0pt}{10pt}},boxrule=.0pt,colframe=pink!50!yellow,
boxsep=0pt,left=6pt,right=6pt,top=0pt,bottom=0pt,enhanced, coltext=white, colback=blue!40!red}

\newtcbox{\xmytboxb}{on line,
arc=7pt,
before upper={\rule[-3pt]{0pt}{10pt}},boxrule=.0pt,colframe=pink!50!yellow,
boxsep=0pt,left=6pt,right=6pt,top=0pt,bottom=0pt,enhanced, coltext=white, colback=white!40!green}

\makeatletter
\newcommand\subsubsubsection{\@startsection{paragraph}{4}{\z@}{-2.5ex\@plus -1ex \@minus -.25ex}{1.25ex \@plus .25ex}{\normalfont\normalsize\bfseries}}
\newcommand\subsubsubsubsection{\@startsection{subparagraph}{5}{\z@}{-2.5ex\@plus -1ex \@minus -.25ex}{1.25ex \@plus .25ex}{\normalfont\normalsize\bfseries}}
\makeatother

\newtheorem{theorem}{Theorem}

\newtheorem{corollary}{Corollary}

\newtheorem{remark}{Remark}

\begin{document}

\begin{singlespace}

\title {RDT based upper bounds on the largest average submatrix values
%\footnote{ This work was supported in
%part.}
}
\author{
\textsc{Mihailo Stojnic
\footnote{e-mail: {\tt flatoyer@gmail.com}}
}}
\date{}
\maketitle

%%%%%%%%%%%%%%%%%%%%%%%%%%%%%%%%%%%%%%%%%%%%%%%%%%%%%%%%%%%%%%%%%%%%%%%%%%%%%%%%
%%%%%%%%%%%%%%%%%%%%%%%%%%%%%%%%%%%%%%%%%%%%%%%%%%%%%%%%%%%%%%%%%%%%%%%%%%%%%%%%
\centerline{{\bf Abstract}} \vspace*{0.1in}
%%%%%%%%%%%%%%%%%%%%%%%%%%%%%%%%%%%%%%%%%%%%%%%%%%%%%%%%%%%%%%%%%%%%%%%%%%%%%%%%
%%%%%%%%%%%%%%%%%%%%%%%%%%%%%%%%%%%%%%%%%%%%%%%%%%%%%%%%%%%%%%%%%%%%%%%%%%%%%%%%

We study statistical variants of the classical largest average submatrix problem. For small submatrices (where the dimension is less than linearly proportional to the original matrix), the problem is typically well understood and believed to exhibit the statistical-computational gap (SCG). However, analytical treatment of both the information-theoretic and algorithmic aspects of the linear regime remains challenging, and no mathematically rigorous results have yet arrived anywhere close to proving or disproving SCG existence in this setting.

Focusing on the linear regime, we make strong progress in several key directions: 1) We develop a generic Random Duality Theory (RDT) framework to characterize largest average submatrix values.
2)  Using the plain RDT variant, we obtain closed-form upper bounds as explicit functions of dimensionality proportionality parameters. 3) We demonstrate that a lifted RDT variant strictly improves upon the plain RDT within a certain linear range of submatrix dimensions. 4)  For small submatrices where the dimensional proportionality constants approach zero, we prove that our results match both the replica results from \cite{ErbaKOZ24} (obtained via one-step replica symmetry breaking) and the sublinear results from \cite{BhamidiDN17, GamarnikLi18}.

\vspace*{0.25in} \noindent {\bf Index Terms: Largest average submatrix; Random duality theory (RDT)}.

\end{singlespace}

%%%%%%%%%%%%%%%%%%%%%%%%%%%%%%%%%%%%%%%%%%%%%%%%%%%%%%%%%%%%%%%%%
%%%%%%%%%%%%%%%%%%%%%%%%%%%%%%%%%%%%%%%%%%%%%%%%%%%%%%%%%%%%%%%%%
%%%%%%%%%%%%%%%%%%%%%%%%%%%%%%%%%%%%%%%%%%%%%%%%%%%%%%%%%%%%%%%%%
%%%%%%%%%%%%%%%%%%%%%%%%%%%%%%%%%%%%%%%%%%%%%%%%%%%%%%%%%%%%%%%%%
%%%%%%%%%%%%%%%%%%%%%%%%%%%%%%%%%%%%%%%%%%%%%%%%%%%%%%%%%%%%%%%%%
%%%%%%%%%%%%%%%%%%%%%%%%%%%%%%%%%%%%%%%%%%%%%%%%%%%%%%%%%%%%%%%%%
\section{Introduction}
\label{sec:back}
%%%%%%%%%%%%%%%%%%%%%%%%%%%%%%%%%%%%%%%%%%%%%%%%%%%%%%%%%%%%%%%%%
%%%%%%%%%%%%%%%%%%%%%%%%%%%%%%%%%%%%%%%%%%%%%%%%%%%%%%%%%%%%%%%%%
%%%%%%%%%%%%%%%%%%%%%%%%%%%%%%%%%%%%%%%%%%%%%%%%%%%%%%%%%%%%%%%%%
%%%%%%%%%%%%%%%%%%%%%%%%%%%%%%%%%%%%%%%%%%%%%%%%%%%%%%%%%%%%%%%%%
%%%%%%%%%%%%%%%%%%%%%%%%%%%%%%%%%%%%%%%%%%%%%%%%%%%%%%%%%%%%%%%%%

Classical notions of computational hardness typically rely on the analysis of worst-case problem instances. In many cases, problems known to be practically hard to solve within worst-case complexity theory are proven to be NP-hard. While correct within the worst-case domain, such a complexity assessment often fails to properly reflect typical algorithmic solvability. Adopting a statistical approach—where typical problem instances are assumed to be probabilistically generated—usually provides a viable alternative that allows for a more adequate capture of algorithmic features.

 Such an approach has flourished over the last couple of decades. Many famous problems have been studied, and excellent results regarding their information theoretic and algorithmic complexity characterizations have been achieved. Examples include: Various types of random perceptrons, asymmetric \cite{KraMez89,GongHLS26,BrZech06,KimRoc98,Bald15,DingSun19,NakSun23,BoltNakSunXu22,Huang24,Gar88,GarDer88,BarbAKZ23,Stojnicalgbp25} and symmetric binary perceptrons \cite{AubPerZde19,PerkXu21, AbbLiSly21b,Stojnicalgsbp26,ElAlGam24,BMPZ23,BanSpen20,GamKizPerXu22,Bald20,Spen85}; 
 Positive \cite{Wendel62,Cover65,Gar88,GarDer88,SchTir03,StojnicGardGen13,Talbook11a,Talbook11b} and negative spherical perceptrons \cite{Gar88,GarDer88,AlaSel20,AMZ24, StojnicGardSphNeg13,FPSUZ17,FraPar16,BMPZ23}; A host of statistical physics models, such as SK and spherical pure or mixed $p$-spin models \cite{GamJag21,Dandietal25,Montanari19,AlaouiMS21}, as well as Hopfield models \cite{Stojniccluphop25}; Many graph-related structures, such as the famous Max-Cut \cite{CGPR19}, along with a plethora of others \cite{GamarSud17,GamAW24,Wein22,GamKW25};
Satisfiability problems \cite{GamarSud17a,MezardPZ02,DinSlySun15}
and numerous other examples, including planted problems (for a non-exhaustive list, see, e.g., \cite{Wein22,HuangS22a,HuangSell24,LiSch24,Gamar21,MMZ05, AchlioptasR06,GamMZ22,BarbierKMMZ18})

In this paper, we consider another classical problem from this group: the so-called largest average submatrix problem (LASP) \cite{ShabalinWPN09,MadOli04,Fort10,PonGAR15,GamarnikLi18,BhamidiDN17}. For a given $m\times n$ matrix $A\in\mR^{m\times n}$, it amounts to finding the submatrix of size $k_1\times k_2$ with the largest average value of its entries $\xi$ \cite{BhaGG26,HegKiz25,ErbaKOZ26,ErbaKOZ24,GamarnikLi18,BhamidiDN17,GamarnikJS21} (see also, e.g., \cite{BrennanBH19,BanksMVVX18,DadonHB24,MaWu15,OrenJBT26,HajekWX18,SohnWein25,SunNobel13} for closely related planted submatrix variants as well as, e.g., \cite{GamKW25,Wein22,GamAW24} for the so-called maximum independent sets problems).

We consider a standard statistical proportional regime (where $m, n, k_1,$ and $k_2$ are all large and proportional to each other) and utilize Random Duality Theory (RDT) to develop a generic framework for analyzing these problems, ultimately determining the upper bounds on $\xi$. To discuss the associated intricacies and relevant prior work in detail, we introduce several technical preliminaries below.

%%%%%%%%%%%%%%%%%%%%%%%%%%%%%%%%%%%%%%%%%%%%%%%%%%%%%%%%%%%%%%%%%
%%%%%%%%%%%%%%%%%%%%%%%%%%%%%%%%%%%%%%%%%%%%%%%%%%%%%%%%%%%%%%%%%
\section{Technical preliminaries, prior work, and contributions}
\label{sec:sampcov}
%%%%%%%%%%%%%%%%%%%%%%%%%%%%%%%%%%%%%%%%%%%%%%%%%%%%%%%%%%%%%%%%%
%%%%%%%%%%%%%%%%%%%%%%%%%%%%%%%%%%%%%%%%%%%%%%%%%%%%%%%%%%%%%%%%%

We will start by introducing precise definitions of the objects that will be the key components of our study.

%%%%%%%%%%%%%%%%%%%%%%%%%%%%%%%%%%%%%%%%%%%%%%%%%%%%%%%%%%%%%%%%%
%%%%%%%%%%%%%%%%%%%%%%%%%%%%%%%%%%%%%%%%%%%%%%%%%%%%%%%%%%%%%%%%%
\subsection{Mathematical setup}
\label{sec:context}
%%%%%%%%%%%%%%%%%%%%%%%%%%%%%%%%%%%%%%%%%%%%%%%%%%%%%%%%%%%%%%%%%
%%%%%%%%%%%%%%%%%%%%%%%%%%%%%%%%%%%%%%%%%%%%%%%%%%%%%%%%%%%%%%%%%

Let $n, m, k_1, k_2 \in \mathbb{N}$ be four positive integers. As stated above, assuming a large $n$, we operate in a linear/proportional dimensional context such that:
\begin{equation}\label{eq:amat0a0}
 \alpha = \lim_{n\rightarrow \infty} \frac{m}{n}, \quad  
  \beta_1 = \lim_{n\rightarrow \infty} \frac{k_1}{n}, \quad 
   \beta_2 = \lim_{n\rightarrow \infty} \frac{k_2}{n}.
 \end{equation}
Setting let $[n] = \{1, 2, \dots, n \}$, for a matrix $A \in \mathbb{R}^{m \times n}$, we are interested in finding its largest average submatrix of size $k_1 \times k_2$ which is   
 \begin{equation}\label{eq:kronrep1}
\xi  = \frac{1}{k_1k_2}\max_{\substack{\cI_1\subset [m],\cI_2\subset [n]\\|\cI_1|=k_1,|\cI_2|=k_2}} \sum_{i\in \cI_1,j\in \cI_2} A_{ij} . 
 \end{equation}
 Moreover, the pair $(\bar{\cI}_1,\bar{\cI}_2)$ defined as
\begin{equation}\label{eq:kronrep1a1}
(\bar{\cI}_1,\bar{\cI}_2)  = \frac{1}{k_1k_2}\mbox{arg} \max_{\substack{\cI_1\subset [m],\cI_2\subset [n]\\|\cI_1|=k_1,|\cI_2|=k_2}} \sum_{i\in \cI_1,j\in \cI_2} A_{ij}, 
 \end{equation}
determines the location of such a submatrix.
 
Similarly to any other optimization problem, two key features of (\ref{eq:kronrep1}) are of interest: (i) the optimal value of the objective ($\xi$), and (ii) the optimal values of the optimizing variables $\cI_1$ and $\cI_2$, given by the pair $(\bar{\cI}_1, \bar{\cI}_2)$. The former is typically associated with the information-theoretic properties of the problem, whereas the latter relates to its algorithmic properties.

In this paper, we focus on the statistical analysis and assume a probabilistic framework. Specifically, we assume that the elements of $A$ are independent standard normal random variables. As noted earlier, a statistical context allows for a more appropriate assessment of the relationship between these two features.

%%%%%%%%%%%%%%%%%%%%%%%%%%%%%%%%%%%%%%%%%%%%%%%%%%%%%%%%%%%%%%%%%%%%%%%%%%%%%%%%%%%%%%%%%%
%%%%%%%%%%%%%%%%%%%%%%%%%%%%%%%%%%%%%%%%%%%%%%%%%%%%%%%%%%%%%%%%%%%%%%%%%%%%%%%%%%%%%%%%%%
%%%%%%%%%%%%%%%%%%%%%%%%%%%%%%%%%%%%%%%%%%%%%%%%%%%%%%%%%%%%%%%%%%%%%%%%%%%%%%%%%%%%%%%%%%
\subsection{Relevant prior work}
\label{sec:comtext}
%%%%%%%%%%%%%%%%%%%%%%%%%%%%%%%%%%%%%%%%%%%%%%%%%%%%%%%%%%%%%%%%%%%%%%%%%%%%%%%%%%%%%%%%%%
%%%%%%%%%%%%%%%%%%%%%%%%%%%%%%%%%%%%%%%%%%%%%%%%%%%%%%%%%%%%%%%%%%%%%%%%%%%%%%%%%%%%%%%%%%
%%%%%%%%%%%%%%%%%%%%%%%%%%%%%%%%%%%%%%%%%%%%%%%%%%%%%%%%%%%%%%%%%%%%%%%%%%%%%%%%%%%%%%%%%%

Following studies on biological and social network applications \cite{ShabalinWPN09,MadOli04,Fort10,PonGAR15} that demonstrated the need for efficient solvers of statistical LASP (or biclustering \cite{MadOli04}), the study of their theoretical and algorithmic properties has flourished over the last decade. Excellent progress has been achieved in many directions, and we briefly survey the milestones most closely related to our work below.

Most prior work focuses on square scenarios. While some situations require these assumptions to facilitate technical analysis, others use them simply to simplify the presentation. To avoid repeating these conditions, we note that the square context generically assumes $n=m$ and $k_1=k_2=k$. Additionally, because the technical challenges differ substantially depending on whether one operates in linear or sublinear regimes, we address each regime and its combinations separately.

\vspace{.1in} 

\noindent \underline{\textbf{\emph{Sublinear regime:}}} In \cite{BhamidiDN17}, the square sublinear regime was considered with $k=O(\log(n))$ to obtain the precise estimate $\xi = 2\sqrt{\log(n)/k}$. Moreover, it was suggested that a simple Largest Average Submatrix (LAS) algorithm from \cite{ShabalinWPN09} should achieve $\xi_{LAS} = \sqrt{2\log(n)/k}$.

\cite{GamarnikLi18} revisited the sublinear regime (with a constant $k$) and confirmed that $\xi_{LAS} = \sqrt{2\log(n)/k}$ is correct. Utilizing a parallel with finding cliques, \cite{GamarnikLi18} first extended the range to $k=o(\log(n))$. They then introduced an Incremental Greedy Procedure (IGP) and showed that it achieves $\xi_{IGP} = \frac{4}{3}\sqrt{2\log(n)/k}$ for $k=o(\log^2(n))$. They also established an Overlap Gap Property (OGP) threshold of $\xi_{OGP} \le \frac{\sqrt{16+2/3}}{3}\sqrt{2\log(n)/k}$.

Drawing parallels with similar graph and satisfiability studies \cite{GamarSud17,Mont15,AchlioptasR06}, \cite{GamarnikLi18} further conjectured that $\xi_{OGP}$ might be related to the algorithmic threshold—a critical value beyond which no polynomial-time algorithm exists. For overviews of OGP's role in algorithmic hardness and statistical-computational gaps, see \cite{GamMZ22,Gamar21}. For the direct implications that various types of OGPs (including multi-, star-, ladder-, and branching-OGPs) have on many well-known problems, see \cite{Montanari19,AlaouiMS21,HuangS22a,HuangSell24,RahVir17,GamarSud17a,GamKizPerXu22,Wein22,BreHuang21}.

\vspace{.1in} 

\noindent \underline{\textbf{\emph{Linear regime:}}} In \cite{ErbaKOZ24}, the authors considered the square linear regime with $k=\beta n$ (noting that the results easily extend to the non-square regime as well). Utilizing statistical physics replica methods, they obtained a wide set of results. For example, they predicted a complete phase transition diagram and determined the maximum average submatrix (MAS) values under both replica symmetry and one-step symmetry-breaking ansatzes.

Moreover, in the $\beta\rightarrow 0$ regime—which is small-dimensional, yet still linearly proportional to the dimension of the original matrix $A$—the appropriately scaled MAS value is estimated as:

$$\lim_{n\rightarrow \infty}\xi\sqrt{n} = 2 \sqrt{\frac{1}{\beta} \log \lp \frac{1}{\beta}\rp }$$

This is believed to be the correct $\beta$ scaling order, and the associated constant 2 is expected to be exact. In fact, extrapolating to the sublinear regimes mentioned in \cite{BhamidiDN17,GamarnikLi18} yields this exact result. Furthermore, \cite{ErbaKOZ24} studied dynamical properties and reached conclusions that align well with the rigorous results in \cite{BhamidiDN17,GamarnikLi18}, despite the latter being established in sublinear regimes.

\vspace{.1in} 

\noindent \underline{\textbf{\emph{Various dimensions and tensors:}}} The linear regime is studied in \cite{HegKiz25} within a tensor context. Assuming a large tensor order $p$, the ground state of the model, $E_{max}$ (corresponding to appropriately scaled $\xi$), is determined. Moving to the sublinear regime $k=o(\log^{1.5}(n))$, \cite{HegKiz25} analyzed the IGP-T algorithm (a tensor generalization of IGP) and showed that it approaches the ground state $\frac{2\sqrt{p}}{p+1}E_{max}$. Furthermore, using a multi-OGP analysis, \cite{HegKiz25} established the presence of a multi-OGP above $\gamma_p E_{max}$, where $\gamma_p \rightarrow \infty$ as $p \rightarrow \infty$. Subsequently, \cite{BhaGG26} reproved the IGP-T performance for $k=o(n)$ and any $p$, and utilized branching-OGP to show that the gap appears above $\frac{2\sqrt{p}}{p+1} E_{max}$. This matches the IGP-T threshold and implies the potential existence of a statistical-computational gap.

Largest average subtensor problems were also considered in \cite{ErbaKOZ26} using replica methods under a $k \ll n$ regime. The primary takeaway is that the matrix conclusions for $p=2$ from \cite{ErbaKOZ24} largely extend to tensors of any order $p > 2$. This applies specifically to the overall structure of the underlying phase diagram and the presence of RS and frozen 1-RSB phases as the largest average submatrix value increases. The performance of LAS and IGP was also analyzed: IGP aligns with the rigorous results of \cite{BhaGG26}, while LAS performance matches the appearance of the frozen 1-RSB phase.

\vspace{.1in} 

\noindent \underline{\textbf{\emph{Planted models:}}}  Closely related to the LASP are planted models, where a specific structure is expected to be hidden a priori within a noisy matrix. Examples include general spiked models \cite{BBP05,PerryWB20,BehneReeves22,LelargeMio18,PakKK23,Lesetal17} and elevated mean models \cite{ButIng13,HajekWX18,BanksMVVX18,CaiLR17}. While LASP focuses on distinguishing the largest average submatrix, planted models aim to infer the hidden structure. These models allow for the study of various topics, including the information-theoretic and algorithmic aspects of partial, approximate, or complete recovery, as well as associated hypothesis testing \cite{LelargeMio18,PakKK23,Lesetal17,MontRich14,DonGavJohn18,BrennanBH19,DadonHB24,MaWu15,OrenJBT26,HajekWX18,SohnWein25,SunNobel13}.

An excellent set of results, particularly related to the linear regime, is obtained in \cite{GamarnikJS21}  (see also, e.g., \cite{CaiLR17,DeshMont14,BarbierMKLZ16,LelargeMio18,KolarNRS11}). \cite{GamarnikJS21}  considers a particular variant of planted LASP that searches for principal submatrices (a symmetrized version of elevated mean submatrix localization). While the remaining problem is still very challenging, restricting the focus to principal submatrices imposes symmetry, which allows for the utilization of established spin glass techniques. As a result, \cite{GamarnikJS21}  achieves a rigorous, precise characterization of $\lim_{n\rightarrow \infty}\xi_{PRIN}\sqrt{n}$.

%%%%%%%%%%%%%%%%%%%%%%%%%%%%%%%%%%%%%%%%%%%%%%%%%%%%%%%%%%%%%%%%%%%%%%%%%%%%%%%%%%%%%%%%%%
%%%%%%%%%%%%%%%%%%%%%%%%%%%%%%%%%%%%%%%%%%%%%%%%%%%%%%%%%%%%%%%%%%%%%%%%%%%%%%%%%%%%%%%%%%
%%%%%%%%%%%%%%%%%%%%%%%%%%%%%%%%%%%%%%%%%%%%%%%%%%%%%%%%%%%%%%%%%%%%%%%%%%%%%%%%%%%%%%%%%%
\subsection{Our contributions}
\label{sec:contr}
%%%%%%%%%%%%%%%%%%%%%%%%%%%%%%%%%%%%%%%%%%%%%%%%%%%%%%%%%%%%%%%%%%%%%%%%%%%%%%%%%%%%%%%%%%
%%%%%%%%%%%%%%%%%%%%%%%%%%%%%%%%%%%%%%%%%%%%%%%%%%%%%%%%%%%%%%%%%%%%%%%%%%%%%%%%%%%%%%%%%%
%%%%%%%%%%%%%%%%%%%%%%%%%%%%%%%%%%%%%%%%%%%%%%%%%%%%%%%%%%%%%%%%%%%%%%%%%%%%%%%%%%%%%%%%%%

As mentioned above, our focus is on non-square statistical scenarios and linearly proportional matrix/submatrix dimensions. More concretely, we consider a matrix $A$ of size $m\times n$ comprised of independent standard normal entries and examine $\mathbb{E} \,\xi\sqrt{n}$—the $\sqrt{n}$-scaled maximum average value of $A$'s submatrices of size $k_1\times k_2$. We assume that $m, n, k_1,$ and $k_2$ satisfy (\ref{eq:amat0a0}).

Below is a brief highlight of some of the main results.

\begin{itemize}
\item Framework Development: We developed a generic framework based on Random Duality Theory (RDT) to characterize largest average submatrix values. This powerful framework accommodates both square and non-square matrices and submatrices, extending beyond the square scenarios prevalent in prior literature.
 
\item Plain RDT Upper-Bounds: Using the plain RDT framework, we obtained explicit, closed-form upper bounds on the maximum average submatrix. These bounds are easy to compute and clearly demonstrate their dependence on key dimensionality-related parameters.
 
\item  Lifted RDT Variant: We introduced a lifted RDT variant that strictly improves upon the plain RDT within a portion of the $\beta\in(0,1)$ range. For $\beta\rightarrow 0$ the improvement is characterized precisely by a factor $\sqrt{2}$.
    
\item  Small Submatrices \& Asymptotic Matching: For small submatrices ($\beta\rightarrow 0$), our results match the replica results from \cite{ErbaKOZ24} obtained via one-step symmetry breaking ansatz. Furthermore, we confirmed that our findings extrapolate to match the sublinear regime results from \cite{BhamidiDN17, GamarnikLi18}. 
       
    \item Algorithmic Implementation:  Finally, the MAS values obtained via practical LAS algorithm from \cite{ShabalinWPN09} are observed to closely approach our theoretical predictions, even for relatively small dimensions ($\sim 1000$). In certain $\beta$ regimes, the difference is $\sim 0.01\%$, suggesting that any statistical-computational gap is either absent or minimal across a sizeable portion of the $\beta\in (0,1)$ range.
\end{itemize}

%%%%%%%%%%%%%%%%%%%%%%%%%%%%%%%%%%%%%%%%%%%%%%%%%%%%%%%%%%%%%%%%%%%%%%%%%%%%%%%%%%%%%%%%%%
%%%%%%%%%%%%%%%%%%%%%%%%%%%%%%%%%%%%%%%%%%%%%%%%%%%%%%%%%%%%%%%%%%%%%%%%%%%%%%%%%%%%%%%%%%
%%%%%%%%%%%%%%%%%%%%%%%%%%%%%%%%%%%%%%%%%%%%%%%%%%%%%%%%%%%%%%%%%%%%%%%%%%%%%%%%%%%%%%%%%%
\section{Upper-bounding $\xi$ via RDT}
\label{sec:xirdt}
%%%%%%%%%%%%%%%%%%%%%%%%%%%%%%%%%%%%%%%%%%%%%%%%%%%%%%%%%%%%%%%%%%%%%%%%%%%%%%%%%%%%%%%%%%
%%%%%%%%%%%%%%%%%%%%%%%%%%%%%%%%%%%%%%%%%%%%%%%%%%%%%%%%%%%%%%%%%%%%%%%%%%%%%%%%%%%%%%%%%%
%%%%%%%%%%%%%%%%%%%%%%%%%%%%%%%%%%%%%%%%%%%%%%%%%%%%%%%%%%%%%%%%%%%%%%%%%%%%%%%%%%%%%%%%%%

We first recall the following  four key RDT principles \cite{StojnicCSetam09,StojnicICASSP10var,StojnicISIT2010binary,StojnicRegRndDlt10} (for more on further upgrades and associated algorithmic implications, see, e.g., \cite{Stojnicalgbp25,Stojnicclupsk25,Stojnicalgsbp26,StojnicNN27,Stojnictcmspnncaprdt23,Stojnictcmspnncapliftedrdt23}):

\begin{enumerate}

\item \emph{Finding underlying optimization algebraic representation}

\item \emph{Determining random dual} 

\item \emph{Handling random dual}

\item \emph{Double-checking strong random duality.}

\end{enumerate}
All four principles, together with their relations to the problem of our interest here, are discussed in detail below.

%%%%%%%%%%%%%%%%%%%%%%%%%%%%%%%%%%%%%%%%%%%%%%%%%%%%%%%%%%%%%%%%%%%%%%%%%%%%%%%%%%%%%%%%%%%%%%%%%%%%
%%%%%%%%%%%%%%%%%%%%%%%%%%%%%%%%%%%%%%%%%%%%%%%%%%%%%%%%%%%%%%%%%%%%%%%%%%%%%%%%%%%%%%%%%%%%%%%%%%%%
\subsection{Finding underlying optimization algebraic representation} 
\label{sec:randpr}
%%%%%%%%%%%%%%%%%%%%%%%%%%%%%%%%%%%%%%%%%%%%%%%%%%%%%%%%%%%%%%%%%%%%%%%%%%%%%%%%%%%%%%%%%%%%%%%%%%%%
%%%%%%%%%%%%%%%%%%%%%%%%%%%%%%%%%%%%%%%%%%%%%%%%%%%%%%%%%%%%%%%%%%%%%%%%%%%%%%%%%%%%%%%%%%%%%%%%%%%%

We start by defining 
\begin{eqnarray}\label{eq:amat0a1}
\cS_{\x} & \triangleq & \left \{\x\in\mR^n \hspace{.05in} |  \hspace{.05in} \|\x\|_2=1, \x_i\in \left \{0,\frac{1}{\sqrt{k_2}} \right \},i=1,\dots,n \right \}
\nonumber \\
& = & 
\left \{\x\in\mR^n \hspace{.05in} |  \hspace{.05in} \x_i\in \left \{0,\frac{1}{\sqrt{k_2}} \right \},i=1,\dots,n, \sum_{i=1}^{n}\x_i=\sqrt{k_2} \right \},\quad \mbox{ and }
\nonumber \\
\cS_{\y} & \triangleq &\left \{\y\in\mR^m \hspace{.05in} |  \hspace{.05in} \|\y\|_2=1, \y_i\in \left \{0,\frac{1}{\sqrt{k_1}} \right \},i=1,\dots,m  \right \} 
\nonumber \\
& = & 
\left \{\y\in\mR^m \hspace{.05in} |  \hspace{.05in} \y_i\in \left \{0,\frac{1}{\sqrt{k_1}} \right \},i=1,\dots,m, \sum_{i=1}^{m}\y_i=\sqrt{k_1} \right \}
.
 \end{eqnarray}
Then we have for the optimal objective from (\ref{eq:kronrep1}) 
\begin{equation}\label{eq:algb1}
\xi  = \frac{1}{\sqrt{k_1k_2}}\max_{\x\in\cS_{\x},\y\in\cS_{\y}} \y^T A \x, 
 \end{equation}
and  for the $(\bar{\x},\bar{\y})  $ analogue of the optimizing pair from (\ref{eq:kronrep1a1})
\begin{equation}\label{eq:algb2}
(\bar{\x},\bar{\y})  = \frac{1}{\sqrt{k_1k_2}}\mbox{argmax}_{\x\in\cS_{\x},\y\in\cS_{\y}} \y^T A \x. 
 \end{equation}
 The optimization problem in (\ref{eq:algb1}) is called \emph{random primal} within the RDT.

%%%%%%%%%%%%%%%%%%%%%%%%%%%%%%%%%%%%%%%%%%%%%%%%%%%%%%%%%%%%%%%%%%%%%%%%%%%%%%%%%%%%%%%%%%%%%%%%%%%%
%%%%%%%%%%%%%%%%%%%%%%%%%%%%%%%%%%%%%%%%%%%%%%%%%%%%%%%%%%%%%%%%%%%%%%%%%%%%%%%%%%%%%%%%%%%%%%%%%%%%
\subsection{Determining random dual} 
\label{sec:randdual}
%%%%%%%%%%%%%%%%%%%%%%%%%%%%%%%%%%%%%%%%%%%%%%%%%%%%%%%%%%%%%%%%%%%%%%%%%%%%%%%%%%%%%%%%%%%%%%%%%%%%
%%%%%%%%%%%%%%%%%%%%%%%%%%%%%%%%%%%%%%%%%%%%%%%%%%%%%%%%%%%%%%%%%%%%%%%%%%%%%%%%%%%%%%%%%%%%%%%%%%%%

The above random primal is complemented in the following theorem with the so-called \emph{random dual}.

\begin{theorem}
\label{thm:thm1}
Consider large dimensional proportional regime with $n,m,k_1,k_2\in\mN$ as in (\ref{eq:amat0a0}). Let each element of $\h\in\mR^m$ and $\g\in\mR^n$ be standard normal random variable independent of any other randomness. Set
\begin{eqnarray}
   \label{eq:thm1eq1}
 L & = &  
\frac{1}{\sqrt{k_1 k_2}} \max_{\x\in\cS_{\x},\y\in\cS_{\y}}  \lp \h^T\y +\g^T\x \rp  .
\end{eqnarray}
For $A\in\mR^{m\times n}$  comprised of independent standard normal entries, let $\xi$ be as in (\ref{eq:algb1}).  One then has  
\begin{eqnarray}
   \label{eq:thm1eq2}
\lim_{n\rightarrow \infty}    \mE \xi\sqrt{n}   \leq  \lim_{n\rightarrow \infty}   \mE L \sqrt{n},
\end{eqnarray}
with the righthand side being the random dual.
\end{theorem}

\begin{proof} For a standard normal $g$ (independent of all other randomness), we consider two centered Gaussian processes indexed by an array $\cX = \{\x,\y\}$
  \begin{eqnarray}
\label{eq:mr1}
 \cG (\cX) & \triangleq &  \cG (\x,\y)  \triangleq   \y^TA\x +  g  \nonumber   \\
 \cG_u (\cX) & \triangleq &  \cG_u (\x,\y)  \triangleq    \h^T\y +\g^T\x.
  \end{eqnarray}
 For two arrays $\cX^{(1)}=\{ \x^{(1)},\y^{(1)}\}$ and $\cX^{(2)}=\{ \x^{(2)},\y^{(2)}\}$ with $ \x^{(i)} \in \cS_{\x}$ and $ \y^{(i)} \in \cS_{\y}$, $i=1,2$,  we further write
  \begin{eqnarray}
\label{eq:mr2}
\mE \cG (\cX^{(1)})\cG (\cX^{(2)})   & =  &   \lp\y^{(1)}\rp^T\y^{(2)} \lp\x^{(2)}\rp^T\x^{(1)} + 1
\nonumber   \\
\mE \cG_u (\cX^{(1)})\cG_u (\cX^{(2)})  & =  & \lp\y^{(2)}\rp^T\y^{(1)} + \lp\x^{(2)}\rp^T\x^{(1)} .
  \end{eqnarray}
From (\ref{eq:mr2}) one then finds
  \begin{align}
\label{eq:mr5}
  \mE & \cG (\cX^{(1)})\cG (\cX^{(2)})
 -
\mE \cG_u (\cX^{(1)})\cG_u (\cX^{(2)} ) =
\nonumber
\\
 & =       
 \lp\y^{(1)}\rp^T\y^{(2)} \lp\x^{(2)}\rp^T\x^{(1)} + 1
 -\lp\y^{(2)}\rp^T\y^{(1)} - \lp\x^{(2)}\rp^T\x^{(1)}
 \nonumber \\
  & =       
\lp 1 -  \lp\y^{(1)}\rp^T\y^{(2)}  \rp \lp 1 - \lp\x^{(2)}\rp^T\x^{(1)}\rp
 \nonumber \\
&  \geq   0.  
 \end{align}
One also notes 
  \begin{align}
\label{eq:mr5a0}
  \mE & \cG (\cX^{(1)})\cG (\cX^{(1)})
 -
\mE \cG_u (\cX^{(1)})\cG_u (\cX^{(1)} )
   = 
\lp 1 -  \lp\y^{(1)}\rp^T\y^{(1)}  \rp \lp 1 - \lp\x^{(1)}\rp^T\x^{(1)}\rp
  =   0,
  \end{align}
  where the last equality follows since  $\|\x^{(1)}\|_2 = \lp \x^{(1)}\rp^T \x^{(1)}=\lp \y^{(1)}\rp^T \y^{(1)}=\|\x^{(1)}\|_2 = 1$.
  
The proof is completed after one recalls Theorem 1.1 from \cite{Gordon85} (the part of the theorem that is utilized below is known as Slepian lemma and was introduced in \cite{Slep62}; both Slepian lemma and the theorem below can be deduced as special cases of results obtained in Corollary 3 in \cite{Stojnicgscompyx16}  and in Corollary 4 in \cite{Stojnicgscomp16}).

\begin{theorem}(\cite{Gordon85,Slep62})
\label{thm:Gordonpos1} Let $X_{i}$ and $Y_{i}$, $1\leq i\leq n$, be two centered Gaussian processes which satisfy the following inequalities for all choices of indices
\begin{enumerate}
\item $\mE(X_{i}^2)=\mE(Y_{i}^2)$
\item $\mE(X_{i}X_{l})\leq \mE(Y_{i}Y_{l}), i\neq l$.
\end{enumerate}
 Then
\begin{equation*}
\mE(\min_{i} X_{i})\leq \mE(\min_i Y_{i}) \quad  \Longleftrightarrow \quad \mE(\max_{i} X_{i})\geq \mE(\max_i Y_{i}).
\end{equation*}
\end{theorem}

After noting correspondence $Y\leftrightarrow\cG$ and $X\leftrightarrow\cG_u$ and
applying  Theorem \ref{thm:Gordonpos1} to processes $\cG(\cdot)$ and  $\cG_u(\cdot)$, we obtain   
\begin{align}\label{eq:mt5a1a0}
&  &\mE \max_{\cX^{(1)}} \cG(\cX)  & \leq \mE \max_{\cX^{(1)}} \cG_u(\cX)
\nonumber \\
\Longleftrightarrow & & \mE \max_{\x\in\cS_{\x},\y\in\cS_{\y}} 
 \lp \y^TA\x +  \g \rp
 & \leq \mE  \max_{\x\in\cS_{\x},\y\in\cS_{\y}}   \lp  \h^T\y + \g^T\x  \rp
\nonumber \\
\Longleftrightarrow & & \mE \max_{\x\in\cS_{\x},\y\in\cS_{\y}} 
  \y^TA\x 
 & \leq \mE  \max_{\x\in\cS_{\x},\y\in\cS_{\y}}   \lp  \h^T\y + \g^T\x  \rp
 .
 \end{align}
Connecting  (\ref{eq:algb1}) and (\ref{eq:mt5a1a0}), one then has
\begin{eqnarray}\label{eq:mt5a1a1}
 \lim_{n\rightarrow \infty}   \sqrt{n}  \mE  \xi  & \leq &   \lim_{n\rightarrow \infty} \frac{\sqrt{n}}{\sqrt{k_1k_2}} \mE  \max_{\x\in\cS_{\x},\y\in\cS_{\y}}   \lp  \h^T\y + \g^T\x  \rp
 ,
 \end{eqnarray}
which, together with (\ref{eq:thm1eq1}), gives  (\ref{eq:thm1eq2}) and completes the proof.
\end{proof}

\begin{remark}
\label{rem:rem0}
To make the presentation neater and writing easier, throughout the paper we focus on expectations. However, all key quantities trivially concentrate, and the stated results automatically extend to hold in a probabilistic sense as well. It is also worth noting that despite being stated in the $n \rightarrow \infty$ regime, the above results hold for any fixed $n, m, k_1,$ and $k_2$.
\end{remark}

%%%%%%%%%%%%%%%%%%%%%%%%%%%%%%%%%%%%%%%%%%%%%%%%%%%%%%%%%%%%%%%%%
%%%%%%%%%%%%%%%%%%%%%%%%%%%%%%%%%%%%%%%%%%%%%%%%%%%%%%%%%%%%%%%%%
\subsection{Handling random dual}
\label{sec:handlerd}
%%%%%%%%%%%%%%%%%%%%%%%%%%%%%%%%%%%%%%%%%%%%%%%%%%%%%%%%%%%%%%%%%
%%%%%%%%%%%%%%%%%%%%%%%%%%%%%%%%%%%%%%%%%%%%%%%%%%%%%%%%%%%%%%%%%

To handle the above random dual, we focus on $L$ as the key object of interest. First, we observe 
\begin{eqnarray}
   \label{eq:hrd1}
 L & \triangleq &  
\frac{1}{\sqrt{k_1 k_2}} \max_{\x\in\cS_{\x},\y\in\cS_{\y}}  \lp \h^T\y +\g^T\x \rp 
   = -  \frac{1}{\sqrt{k_1 k_2}} \min_{\x\in\cS_{\x},\y\in\cS_{\y}}  \lp -\h^T\y -\g^T\x \rp  .
\end{eqnarray}
Since we are ultimately interested in the statistical properties of $L$, we adopt a cosmetic change of variables, $\h\rightarrow -\h$ and $\g\rightarrow -\g$. Due to symmetry, $\h$ and $\g$ remain independent standard normal vectors. For $\nu_1$ and $\nu_2$, the Lagrangian is expressed as follows:
 \begin{eqnarray}
   \label{eq:hrd2}
 \cL  =  \h^T\y +\g^T\x   + \nu_1\sum_{i=1}^{m} \y_i - \nu_1\sqrt{k_1}  + \nu_2\sum_{i=1}^{n} \x_i - \nu_2\sqrt{k_2} .
\end{eqnarray}
A combination of (\ref{eq:hrd1}) and  (\ref{eq:hrd2}) together with the strong duality gives
\begin{eqnarray}
   \label{eq:hrd3}
 L  =  -  \frac{1}{\sqrt{k_1 k_2}} \min_{\x_i\in \left \{0,\frac{1}{\sqrt{k_2}} \right \},\y_i\in \left \{0,\frac{1}{\sqrt{k_1}} \right \}}  \max_{\nu_1,\nu_2} \cL
 = 
 -  \frac{1}{\sqrt{k_1 k_2}}   \max_{\nu_1,\nu_2}  \min_{\x_i\in \left \{0,\frac{1}{\sqrt{k_2}} \right \},\y_i\in \left \{0,\frac{1}{\sqrt{k_1}} \right \}} \cL.
\end{eqnarray}
After solving the inner minimization, we obtain
\begin{eqnarray}
   \label{eq:hrd4}
 L 
  & = &   
 -  \frac{1}{\sqrt{k_1 k_2}}   \max_{\nu_1,\nu_2} 
\lp
 \lp \frac{1}{\sqrt{k_1}}\sum_{i=1}^{m} \min (\h_i+\nu_1,0) \rp -\nu_1\sqrt{k_1}
 +
 \lp \frac{1}{\sqrt{k_2}}\sum_{i=1}^{n} \min (\g_i+\nu_2,0) \rp
-\nu_2\sqrt{k_2}
 \rp
 \nonumber \\
  & = &   
  \frac{1}{k_1\sqrt{ k_2}}   \min_{\nu_1} 
\lp -\sum_{i=1}^{m} \min (\h_i+\nu_1,0) +\nu_1k_1 \rp
 +
  \frac{1}{k_2\sqrt{ k_1}}   \min_{\nu_2} 
 \lp - \sum_{i=1}^{n} \min (\g_i+\nu_2,0) +\nu_2 k_2 \rp
. 
   .\end{eqnarray}
 The law of large numbers and concentrations give
\begin{align}
   \label{eq:hrd8}
\lim_{n\rightarrow\infty}  \sqrt{n} \mE L   
& =  \lim_{n\rightarrow\infty}   \min_{\nu_1} 
 \lp \frac{m\sqrt{n}}{k_1\sqrt{ k_2}}   \mE \lp -  \min (\h_1+\nu_1,0) \rp + \nu_1\frac{\sqrt{n}}{\sqrt{k_2}}  \rp
\nonumber \\
& \hspace{.15in} +
\lim_{n\rightarrow\infty}   \min_{\nu_2} 
\lp  \frac{n\sqrt{n}}{k_2\sqrt{ k_1}}   \mE \lp -  \min (\g_1+\nu_2,0) \rp  + \nu_2\frac{\sqrt{n}}{\sqrt{k_1}} \rp
\nonumber \\
& =      \min_{\nu_1} 
\lp \frac{\alpha }{\beta_1\sqrt{ \beta_2}}  \mE \lp -  \min (\h_1+\nu_1,0) \rp + \frac{\nu_1}{\sqrt{\beta_2}} \rp
\nonumber \\
&  \hspace{.15in} +
   \min_{\nu_2} 
 \lp \frac{1}{\beta_2\sqrt{ \beta_1}}  \mE \lp -  \min (\g_1+\nu_2,0) \rp + \frac{\nu_2}{\sqrt{\beta_1}}
 \rp
\nonumber \\
& =    
\frac{1}{\sqrt{\beta_1\beta_2}}  \min_{\nu_1,\nu_2} 
\lp \frac{\alpha }{\sqrt{ \beta_1}}  \mE \lp -  \min (\h_1+\nu_1,0) \rp + \nu_1\sqrt{\beta_1} 
  +
   \frac{1}{\sqrt{ \beta_2}}  \mE \lp -  \min (\g_1+\nu_2,0) \rp + \nu_2 \sqrt{\beta_2}
 \rp
 . 
\end{align}
 After setting
\begin{equation}\label{eq:hrd9}
 I(\nu) = \frac{1}{2} \nu \hspace{.03in}\erfc\lp \frac{\nu}{\sqrt{2} }\rp - \frac{e^{-\frac{\nu^2}{2}}}{\sqrt{2 \pi} },
\end{equation} 
and solving the integral one finds
\begin{align}
   \label{eq:hrd10}
\lim_{n\rightarrow\infty}  \sqrt{n} \mE L   
  =    
\frac{1}{\sqrt{\beta_1\beta_2}}  \min_{\nu_1,\nu_2} 
\omega(\nu_1,\nu_2) , 
\end{align}
where
\begin{align}
   \label{eq:hrd11}
\omega(\nu_1,\nu_2) 
= 
\lp  -\frac{\alpha }{\sqrt{ \beta_1}}   I(\nu_1) 
  + \nu_1\sqrt{\beta_1}
  - \frac{1}{\sqrt{ \beta_2}} I( \nu_2) 
  + \nu_2\sqrt{\beta_2}   \rp
 . 
\end{align}
Computing the derivatives over $\nu$ gives
 \begin{equation}\label{eq:hrd12}
 \frac{dI(\nu)}{d\nu} =   \frac{\mathrm{erfc}\left(\frac{\sqrt{2}\,\nu }{2}\right)}{2}-\frac{\sqrt{2}\,\nu \,{\mathrm{e}}^{-\frac{\nu ^2}{2}}}{2\,\sqrt{\pi }}+\frac{\sqrt{2}\,\nu \,{\mathrm{e}}^{-\frac{\nu ^2}{2}}}{2\,\sqrt{\pi }}
 =
 \frac{\mathrm{erfc}\left(\frac{\sqrt{2}\,\nu }{2}\right)}{2},
 \end{equation}
 and
 \begin{equation}\label{eq:hrd12a0}
 \frac{d^2I(\nu)}{d\nu^2} =   
\frac{d\lp  \frac{\mathrm{erfc}\left(\frac{\sqrt{2}\,\nu }{2}\right)}{2} \rp}{d\nu}
=
 -\frac{\sqrt{2}\,{\mathrm{e}}^{-\frac{\nu ^2}{2}}}{2\,\sqrt{\pi }} < 0
 .
 \end{equation}
 From  (\ref{eq:hrd11}) and (\ref{eq:hrd12}), we find
 \begin{equation}\label{eq:hrd13}
 \frac{d\omega(\nu_1,\nu_2)}{d\nu_1} =
-\frac{\alpha}{\sqrt{\beta_1}} \frac{dI(\nu_1)}{d\nu_1} +\sqrt{\beta_1} 
 =
-\frac{\alpha}{\sqrt{\beta_1}}  \frac{\mathrm{erfc}\left(\frac{\sqrt{2}\,\nu_1 }{2}\right)}{2}
+\sqrt{\beta_1} .
 \end{equation}
 Moreover,  from (\ref{eq:hrd12a0}) and (\ref{eq:hrd13}), one also have 
 \begin{equation}\label{eq:hrd13a0}
 \frac{d^2\omega(\nu_1,\nu_2)}{d\nu_1^2} =
-\frac{\alpha}{\sqrt{\beta_1}} \frac{d^2I(\nu_1)}{d\nu_1^2}  =
 \frac{\alpha}{\sqrt{\beta_1}} 
 \frac{\sqrt{2}\,{\mathrm{e}}^{-\frac{\nu_1 ^2}{2}}}{2\,\sqrt{\pi }} > 0.
 \end{equation}
 A combination of (\ref{eq:hrd13}), and (\ref{eq:hrd13a0}) then give $\hat{\nu}_1$ as optimal $\nu_1$
 \begin{equation}\label{eq:hrd14}
 \hat{\nu}_1 =
\sqrt{2}\erfinv\lp 1- \frac{2\beta_1}{\alpha} \rp.
 \end{equation}
Analogously to (\ref{eq:hrd13}) and (\ref{eq:hrd14}), we also have
 \begin{equation}\label{eq:hrd15}
 \frac{d\omega(\nu_1,\nu_2)}{d\nu_2} =
-\frac{1}{\sqrt{\beta_2}} \frac{dI(\nu_1)}{d\nu_2} +\sqrt{\beta_2} 
 =
-\frac{1}{\sqrt{\beta_2}}  \frac{\mathrm{erfc}\left(\frac{\sqrt{2}\,\nu_2 }{2}\right)}{2}
+\sqrt{\beta_2} ,
 \end{equation}
and   $\hat{\nu}_2$ as optimal $\nu_2$
 \begin{equation}\label{eq:hrd16}
 \hat{\nu}_2 =
\sqrt{2}\erfinv\lp 1- 2\beta_2\rp.
 \end{equation}
 
We summarize the above in the following theorem.

\begin{theorem}
\label{thm:thm2}
 Assume the setup of Theorem \ref{thm:thm1}. Then
 \begin{eqnarray}
   \label{eq:thm2eq2}
\lim_{n\rightarrow \infty}    \mE \xi\sqrt{n}  & \leq & \lim_{n\rightarrow \infty}   \mE L \sqrt{n}
=
\frac{1}{\sqrt{\beta_1\beta_2}}
\omega(\hat{\nu}_1,\hat{\nu}_2) 
\nonumber \\
& = & 
\frac{1}{\sqrt{\beta_1\beta_2}}
\lp  \frac{\alpha }{\sqrt{ \beta_1}} \frac{e^{-\lp \erfinv\lp1-\frac{2\beta_1}{\alpha}\rp\rp^2}}{\sqrt{2\pi}} 
   + 
    \frac{1 }{\sqrt{ \beta_2}} \frac{e^{-\lp \erfinv(1-2\beta_2)\rp^2}}{\sqrt{2\pi}}
   \rp .
\end{eqnarray}
 \end{theorem}

\begin{proof}
Follows from the previous discussion through a combination of (\ref{eq:thm1eq2}), (\ref{eq:hrd10}),  (\ref{eq:hrd11}),  (\ref{eq:hrd14}), and  (\ref{eq:hrd16}). In particular, one notes from (\ref{eq:hrd9}) and (\ref{eq:hrd14})
\begin{align}\label{eq:pr2eq1}
 I(\hat{\nu}_1) & = \frac{1}{2} \hat{\nu}_1 \hspace{.03in}\erfc\lp \frac{\hat{\nu}_1}{\sqrt{2} }\rp - \frac{e^{-\frac{\hat{\nu}_1^2}{2}}}{\sqrt{2 \pi} }
 =
 \frac{1}{2} \hat{\nu}_1 \lp 1 - \erf\lp \frac{\hat{\nu}_1}{\sqrt{2} }\rp \rp- \frac{e^{-\frac{\hat{\nu}_1^2}{2}}}{\sqrt{2 \pi} }
\nonumber \\
& =
 \frac{1}{2} \hat{\nu}_1 \lp 1 - \erf\lp \erfinv\lp 1- \frac{2\beta_1}{\alpha}\rp\rp \rp- \frac{e^{-\frac{\hat{\nu}_1^2}{2}}}{\sqrt{2 \pi} }
 =
 \frac{\beta_1}{\alpha} \hat{\nu}_1  - \frac{e^{-\frac{\hat{\nu}_1^2}{2}}}{\sqrt{2 \pi} } ,
\end{align}
and analogously from (\ref{eq:hrd9}) and (\ref{eq:hrd16}) 
\begin{align}\label{eq:pr2eq2}
 I(\hat{\nu}_2) & =  
  \beta_2 \hat{\nu}_2  - \frac{e^{-\frac{\hat{\nu}_2^2}{2}}}{\sqrt{2 \pi} } .
\end{align}
Plugging $ I(\hat{\nu}_1) $ from (\ref{eq:pr2eq1}) and $ I(\hat{\nu}_2)$  from (\ref{eq:pr2eq2}) into 
(\ref{eq:hrd11}) gives
\begin{align}
   \label{eq:pr2eqe}
\frac{1}{\sqrt{\beta_1\beta_2}}\omega(\hat{\nu}_1,\hat{\nu}_2) 
& =
\frac{1}{\sqrt{\beta_1\beta_2}} 
\lp  -\frac{\alpha }{\sqrt{ \beta_1}} 
\lp   \frac{\beta_1}{\alpha} \hat{\nu}_1  - \frac{e^{-\frac{\hat{\nu}_1^2}{2}}}{\sqrt{2 \pi} }   \rp
  + \hat{\nu}_1\sqrt{\beta_1}
  - \frac{1}{\sqrt{ \beta_2}} \lp   \beta_2 \hat{\nu}_2 
   - \frac{e^{-\frac{\hat{\nu}_2^2}{2}}}{\sqrt{2 \pi} }\rp 
  + \hat{\nu}_2\sqrt{\beta_2}   \rp
\nonumber \\
& =
\frac{1}{\sqrt{\beta_1\beta_2}} 
\lp  \frac{\alpha }{\sqrt{ \beta_1}} 
   \frac{e^{-\frac{\hat{\nu}_1^2}{2}}}{\sqrt{2 \pi} }  
  + \frac{1}{\sqrt{ \beta_2}}  \frac{e^{-\frac{\hat{\nu}_2^2}{2}}}{\sqrt{2 \pi} }  \rp
\nonumber \\
& =
\frac{1}{\sqrt{\beta_1\beta_2}} 
\lp  \frac{\alpha }{\sqrt{ \beta_1}} 
   \frac{e^{-\lp \erfinv\lp 1- \frac{2\beta_1}{\alpha}\rp\rp^2}}{\sqrt{2 \pi} }  
  + \frac{1}{\sqrt{ \beta_2}}  \frac{e^{-\lp\erfinv\lp 1-2\beta_2 \rp\rp^2}}{\sqrt{2 \pi} }  \rp
 , 
\end{align}
which matches (\ref{eq:thm2eq2}) and completes the proof.
\end{proof}

We also have the following corollary in the squared scenario. 

\begin{corollary}
\label{cor:cor1}
 Assume the squared matrices setup of Theorem \ref{thm:thm2} with $\alpha=1$ and $\beta_1=\beta_2=\beta$. Then
 \begin{eqnarray}
   \label{eq:cor1eq2}
\lim_{n\rightarrow \infty}    \mE \xi\sqrt{n}  & \leq &   
     \sqrt{\frac{2}{\pi\beta^3}}  e^{-\lp \erfinv(1-2\beta)\rp^2} 
   .
\end{eqnarray}
 \end{corollary}

\begin{proof}
Follows automatically from Theorem \ref{thm:thm2}.
\end{proof}

We also have another corollary that addresses small submatrices within the squared scenario. 

\begin{corollary}
\label{cor:cor1a0}
As in Corollary \ref{cor:cor1}, assume the squared matrices setup of Theorem \ref{thm:thm2} with $\alpha=1$ and $\beta_1=\beta_2=\beta$. Additionally, let $\beta\rightarrow 0$. Then
 \begin{eqnarray}
   \label{eq:cor1eq2}
\lim_{n\rightarrow \infty}    \mE \xi\sqrt{n}  & \leq &   
     2  \sqrt{\frac{2}{\beta}\log\lp \frac{1}{\beta}  \rp   }   .
\end{eqnarray}
 \end{corollary}

\begin{proof}
We first set
\begin{eqnarray}
 \label{eq:rseq1}
   A &=& \erfinv(1-2\beta).
\end{eqnarray}
Then
\begin{eqnarray}
\label{eq:rseq2}
\beta =\frac{1}{2}\erfc\lp A\rp.
\end{eqnarray}
Since $\beta \rightarrow 0$ one then has $A\rightarrow \infty$. From (\ref{eq:rseq2} we then have)
\begin{eqnarray}
 \label{eq:rseq2}
\beta =\frac{1}{2}\erfc\lp A\rp \longrightarrow 
\frac{1}{2\sqrt{\pi}} \frac{e^{-A^2}}{A} \quad \mbox{and}\quad A\longrightarrow \sqrt{\log\lp \frac{1}{\beta} \rp} .
\end{eqnarray}

This then gives
\begin{eqnarray}
 \label{eq:rseq3}
     \sqrt{\frac{2}{\pi\beta^3}}  e^{-\lp \erfinv(1-2\beta)\rp^2} 
     =
          \sqrt{\frac{2}{\pi\beta^3}}  e^{-A^2} \longrightarrow \sqrt{\frac{2}{\beta^3}}2A\beta
 \longrightarrow \sqrt{\frac{2}{\beta^3}}2   \sqrt{ \log\lp \frac{1}{\beta} \rp } \beta     
  \longrightarrow   2\sqrt{ \frac{2}{\beta}  \log\lp \frac{1}{\beta} \rp }  .
\end{eqnarray}
 
\end{proof}

%%%%%%%%%%%%%%%%%%%%%%%%%%%%%%%%%%%%%%%%%%%%%%%%%%%%%%%%%%%%%%%%%
%%%%%%%%%%%%%%%%%%%%%%%%%%%%%%%%%%%%%%%%%%%%%%%%%%%%%%%%%%%%%%%%%
\subsection{Double-checking strong random duality}
\label{sec:strrd}
%%%%%%%%%%%%%%%%%%%%%%%%%%%%%%%%%%%%%%%%%%%%%%%%%%%%%%%%%%%%%%%%%
%%%%%%%%%%%%%%%%%%%%%%%%%%%%%%%%%%%%%%%%%%%%%%%%%%%%%%%%%%%%%%%%%

The random dual  established through Theorems \ref{thm:thm1}  and \ref{thm:thm2}   upper-bounds $\mE\xi$.  Finally, the last step of the RDT machinery relats to double checking if the strong random
duality holds as well. As the underlying problem is of highly discrete nature, the corresponding reversal considerations from \cite{StojnicRegRndDlt10} can not be applied suggesting the absence of strong random duality and the strict boundedness nature of the above upper bounds. The next section shows that this is indeed the case.

%%%%%%%%%%%%%%%%%%%%%%%%%%%%%%%%%%%%%%%%%%%%%%%%%%%%%%%%%%%%%%%%%
%%%%%%%%%%%%%%%%%%%%%%%%%%%%%%%%%%%%%%%%%%%%%%%%%%%%%%%%%%%%%%%%%
\section{Lowering upper bounds via lifted RDT}
\label{sec:plrdt}
%%%%%%%%%%%%%%%%%%%%%%%%%%%%%%%%%%%%%%%%%%%%%%%%%%%%%%%%%%%%%%%%%
%%%%%%%%%%%%%%%%%%%%%%%%%%%%%%%%%%%%%%%%%%%%%%%%%%%%%%%%%%%%%%%%%

To see whether the above bounds can be further lowered, we utilize the lifted methodology \cite{Stojnictcmspnncapliftedrdt23}. In particular we have the following theorem.

\begin{theorem}
\label{thm:thm4}
Following the Theorem \ref{thm:thm1} setup, let $n,m,k_1,k_2\in\mN$ be as in (\ref{eq:amat0a0}) and let each element of $\h\in\mR^m$ and $\g\in\mR^n$ be standard normal random variable independent of any other randomness. For $c_3>0$ set
\begin{eqnarray}
   \label{eq:thm4eq1}
 L^{(2)} & = &  
\frac{1}{\sqrt{k_1 k_2}} 
\min_{c_3>0}\lp 
  -\frac{c_{3}}{2} + \frac{1}{c_3}\log \lp
    \mE e^{c_{3}  \max_{\y\in\cS_{\y}}   \lp  \h^T\y  \rp    }
        \rp
        +
  \frac{1}{c_3}\log \lp
    \mE e^{c_{3}  \max_{\x\in\cS_{\x}}   \lp  \g^T\x \rp    }
        \rp
\rp .
\end{eqnarray}
For $A\in\mR^{m\times n}$  comprised of independent standard normal entries and $\xi$ as in (\ref{eq:algb1}), one has  
\begin{eqnarray}
   \label{eq:thm4eq2}
\lim_{n\rightarrow \infty}    \mE \xi\sqrt{n}   \leq  \lim_{n\rightarrow \infty}   \mE L^{(2)} \sqrt{n} \leq  \lim_{n\rightarrow \infty}   \mE L \sqrt{n},
\end{eqnarray}
with the middle term being the lifted random dual.
\end{theorem}

\begin{proof}
We start by recalling Corollary 1.3 from \cite{Gordon85}.   

\begin{theorem}(\cite{Gordon85})
\label{thm:Gordoncor1} Assume the setup of Theorem \ref{thm:Gordonpos1}. Let $\psi_q(\cdot)$ be an increasing function on the real axis. Then
\begin{equation}\label{eq:thm5eq1}
E(\min_{i}\psi_q(X_{i}))\leq E(\min_i \psi_q(Y_{i})) \quad \Longleftrightarrow \quad  E(\max_{i}\psi_q(X_{i}))\geq E(\max_i\psi_q(Y_{i})).
\end{equation}
\end{theorem}

For $\psi_q(x)=x$, one has that Theorem \ref{thm:Gordoncor1} matches Theorem \ref{thm:Gordonpos1}. The key question is whether a better than linear choice can be found for $\psi_q(\cdot)$. We discover that $\psi_q(x)=e^{c_3 x}$, with $c_3>0$, is an excellent choice. Along the same lines, we also note that Theorems \ref{thm:Gordonpos1} and  \ref{thm:Gordoncor1} are special cases of concepts presented in Corollary 3 in \cite{Stojnicgscompyx16}  and in Corollary 4 in \cite{Stojnicgscomp16}.

Combining (\ref{eq:mr5}), (\ref{eq:mr5a0}), and (\ref{eq:thm5eq1})  with correspondence $Y\leftrightarrow\cG$ and $X\leftrightarrow\cG_u$ and the above mentioned choice $\psi_q(x)=e^{c_3 x},c_3>0$, gives the following
\begin{eqnarray}
\label{eq:mr6}
  \mE \max_{\x\in\cS_{\x},\y\in\cS_{\y}} e^{c_{3,u}\cG(\x,\y)} &\leq &    \mE \max_{\x\in\cS_{\x},\y\in\cS_{\y}} e^{c_{3,u}\cG_u(\x,\y)} .
\end{eqnarray}
From (\ref{eq:mr1}) and (\ref{eq:mr6}) we then find
\begin{eqnarray}
\label{eq:mr7}
  \mE  \max_{\x\in\cS_{\x},\y\in\cS_{\y}} e^{c_{3}\lp  \y^TA\x +  g \rp } &\leq &   \mE \max_{\x\in\cS_{\x},\y\in\cS_{\y}}  e^{c_{3} \lp  \h^T\y +\g^T\x \rp  },
  \end{eqnarray}
which is equivalent to
\begin{eqnarray}
\label{eq:mr8}
  \mE  e^{c_{3} \max_{\x\in\cS_{\x},\y\in\cS_{\y}}   \lp  \y^TA\x \rp } 
  \leq  
   e^{-\frac{c_{3}^2}{2}} \mE e^{c_{3}  \max_{\x\in\cS_{\x},\y\in\cS_{\y}}   \lp  \h^T\y +\g^T\x \rp    } ,
\end{eqnarray}
and
\begin{eqnarray}
\label{eq:mr9}
     \log \lp \mE  e^{c_{3} \max_{\x\in\cS_{\x},\y\in\cS_{\y}}   \lp  \y^TA\x \rp } 
   \rp
      \leq 
   -\frac{c_{3}^2}{2} + \log \lp
    \mE e^{c_{3}  \max_{\x\in\cS_{\x},\y\in\cS_{\y}}   \lp  \h^T\y +\g^T\x \rp    }
        \rp.
\end{eqnarray}
Applying Jensen's inequality to the left-hand side and dividing both sides by $c_3$, we obtain
\begin{eqnarray}
\label{eq:mr10}
\mE \max_{\x\in\cS_{\x},\y\in\cS_{\y}}    \y^TA\x   
     & \leq &
   -\frac{c_{3}}{2} + \frac{1}{c_3}\log \lp
    \mE e^{c_{3}  \max_{\x\in\cS_{\x},\y\in\cS_{\y}}   \lp  \h^T\y +\g^T\x \rp    }
        \rp
\nonumber \\
 & \leq &
  -\frac{c_{3}}{2} + \frac{1}{c_3}\log \lp
    \mE e^{c_{3}  \max_{\y\in\cS_{\y}}   \lp  \h^T\y  \rp    }
        \rp
        +
  \frac{1}{c_3}\log \lp
    \mE e^{c_{3}  \max_{\x\in\cS_{\x}}   \lp  \g^T\x \rp    }
        \rp.
\end{eqnarray}
Since the above holds for any $c_3>0$, it also holds for the ``worst-case'' one, i.e., one has
\begin{eqnarray}
\label{eq:mr10}
\mE \max_{\x\in\cS_{\x},\y\in\cS_{\y}}    \y^TA\x   
   \leq 
   \min_{c_3>0} 
   \lp
  -\frac{c_{3}}{2} + \frac{1}{c_3}\log \lp
    \mE e^{c_{3}  \max_{\y\in\cS_{\y}}   \lp  \h^T\y  \rp    }
        \rp
        +
  \frac{1}{c_3}\log \lp
    \mE e^{c_{3}  \max_{\x\in\cS_{\x}}   \lp  \g^T\x \rp    }
        \rp
        \rp ,
\end{eqnarray}
which matches the first inequality in (\ref{eq:thm4eq1}). After observing that the second inequality in (\ref{eq:thm4eq1}) follows automatically for $c_3\rightarrow 0$, the theorem's proof is completed.
\end{proof}

%%%%%%%%%%%%%%%%%%%%%%%%%%%%%%%%%%%%%%%%%%%%%%%%%%%%%%%%%%%%%%%%%
%%%%%%%%%%%%%%%%%%%%%%%%%%%%%%%%%%%%%%%%%%%%%%%%%%%%%%%%%%%%%%%%%
\subsection{Handling lifted random dual}
\label{sec:hlrd}
%%%%%%%%%%%%%%%%%%%%%%%%%%%%%%%%%%%%%%%%%%%%%%%%%%%%%%%%%%%%%%%%%
%%%%%%%%%%%%%%%%%%%%%%%%%%%%%%%%%%%%%%%%%%%%%%%%%%%%%%%%%%%%%%%%%

From (\ref{eq:hrd2})-(\ref{eq:hrd4}) we have
\begin{eqnarray}
   \label{eq:hlrd1}
 \max_{\y\in\cS_{\y}}   \lp  \h^T\y  \rp
  & = &   
\min_{\nu_1} \lp  -    \frac{1}{\sqrt{k_1}}\sum_{i=1}^{m} \min (\h_i+\nu_1,0)  +\nu_1\sqrt{k_1} \rp
\nonumber \\
\max_{\x\in\cS_{\x}}   \lp  \g^T\x  \rp  & = &   
\min_{\nu_2} \lp  - \frac{1}{\sqrt{k_2}}\sum_{i=1}^{n} \min (\g_i+\nu_2,0) 
+\nu_2\sqrt{k_2} \rp
    .
\end{eqnarray}
We then also have
\begin{eqnarray}
   \label{eq:hlrd2}
\frac{1}{c_3}\log \lp
    \mE e^{c_{3}  \max_{\y\in\cS_{\y}}   \lp  \h^T\y  \rp    }
        \rp
&  = &
\min_{\nu_1} 
\frac{1}{c_3}\log \lp
    \mE e^{c_{3}  \lp  -    \frac{1}{\sqrt{k_1}}\sum_{i=1}^{m} \min (\h_i+\nu_1,0)  +\nu_1\sqrt{k_1}
   \rp    }
        \rp
\nonumber \\
& = &        
\min_{\nu_1}  
\lp
\frac{m}{c_3}\log \lp
    \mE e^{c_{3}  \lp  -    \frac{1}{\sqrt{k_1}}\min (\h_1+\nu_1,0)  
   \rp    }
        \rp + \nu_1\sqrt{k_1} \rp,
\end{eqnarray}
and
\begin{eqnarray}
   \label{eq:hlrd3}
\frac{1}{c_3}\log \lp
    \mE e^{c_{3}  \max_{\x\in\cS_{\x}}   \lp  \g^T\x  \rp    }
        \rp
&  = &
\min_{\nu_2} 
\frac{1}{c_3}\log \lp
    \mE e^{c_{3}  \lp  -    \frac{1}{\sqrt{k_2}}\sum_{i=1}^{n} \min (\g_i+\nu_2,0)  +\nu_2\sqrt{k_2}
   \rp    }
        \rp
\nonumber \\
& = &      
\min_{\nu_2} 
\lp    
\frac{n}{c_3}\log \lp
    \mE e^{c_{3}  \lp  -    \frac{1}{\sqrt{k_2}}\min (\g_1+\nu_2,0)  
   \rp    }
        \rp + \nu_2\sqrt{k_2} \rp.
\end{eqnarray}
Combining (\ref{eq:thm4eq1}), (\ref{eq:hlrd2}), and (\ref{eq:hlrd3}),  we further find
\begin{multline}
\label{eq:hlrd4}
L^{(2)}   = \frac{1}{\sqrt{k_1k_2}} 
   \min_{c_3>0,\nu_1,\nu_2} 
   \Bigg (\Bigg .
  -\frac{c_{3}}{2} 
  + 
  \frac{m}{c_3}\log \lp
    \mE e^{c_{3}  \lp  -    \frac{1}{\sqrt{k_1}}\min (\h_1+\nu_1,0)  
   \rp    }
        \rp + \nu_1\sqrt{k_1} 
\\
        +
\frac{n}{c_3}\log \lp
    \mE e^{c_{3}  \lp  -    \frac{1}{\sqrt{k_2}}\min (\g_1+\nu_2,0)  
   \rp    }
        \rp + \nu_2\sqrt{k_2}         \Bigg .\Bigg ) .
\end{multline}
 Adopting $c_3\rightarrow c_3\sqrt{n}$ scaling, one obtains
\begin{multline}
 \label{eq:hlrd5}
\lim_{n\rightarrow \infty } \sqrt{n} L^{(2)}   = 
\frac{1}{\sqrt{\beta_1 \beta_2}} 
   \min_{c_3>0,\nu_1,\nu_2} 
   \Bigg (\Bigg .
  -\frac{c_{3}}{2} 
  + 
  \frac{\alpha}{c_3}\log \lp
    \mE e^{c_{3}  \lp  -    \frac{1}{\sqrt{\beta_1}}\min (\h_1+\nu_1,0)  
   \rp    }
        \rp + \nu_1\sqrt{\beta_1} 
\\
        +
\frac{1}{c_3}\log \lp
    \mE e^{c_{3}  \lp  -    \frac{1}{\sqrt{\beta_2}}\min (\g_1+\nu_2,0)  
   \rp    }
        \rp + \nu_2\sqrt{\beta_2}         \Bigg .\Bigg ) .
\end{multline}
After solving the integrals, we also have
 \begin{eqnarray}
 \label{eq:hlrd6}
 I^{(2)}(c_3,\nu;\beta) &=& \mE e^{c_{3}  \lp  -    \frac{1}{\sqrt{\beta}}\min (\h_1+\nu,0)  
   \rp    } 
    =  \frac{\mathrm{erfc}\left(-\frac{\sqrt{2}\,\nu }{2}\right)}{2}
      +\frac{{\mathrm{e}}^{-\frac{c_{3}\,\left(2\,\nu -\frac{c_{3}}{\sqrt{\beta }}\right)}{2\,\sqrt{\beta }}}\,\mathrm{erfc}\left(\frac{\sqrt{2}\,\left(\nu -\frac{c_{3}}{\sqrt{\beta }}\right)}{2}\right)}{2} .
 \end{eqnarray}
A combination of (\ref{eq:hlrd5}) and (\ref{eq:hlrd6}) gives
\begin{multline}
\label{eq:hlrd7}
\lim_{n\rightarrow \infty } \sqrt{n} L^{(2)}   = 
   \frac{1}{\sqrt{\beta_1 \beta_2}} 
   \min_{c_3>0,\nu_1,\nu_2} 
   \Bigg (\Bigg .
  -\frac{c_{3}}{2} 
  + 
  \frac{\alpha}{c_3}\log \lp
    I^{(2)}(c_3,\nu_1;\beta_1)
        \rp + \nu_1\sqrt{\beta_1} 
       \\
        +
\frac{1}{c_3}\log \lp
       I^{(2)}(c_3,\nu_2;\beta_2)
      \rp + \nu_2\sqrt{\beta_2}         \Bigg .\Bigg ) .
\end{multline}
  
We summarize the above discussion into the following theorem.

\begin{theorem}
\label{thm:thm6}
Assume the setup of Theorem \ref{thm:thm4} with $\alpha,\beta_1,$ and $\beta_2$ as in (\ref{eq:amat0a0}). Let $ I^{(2)}(c_3,\nu;\beta) $ be as in (\ref{eq:hlrd6}) and $L^{(2)}$ as in
(\ref{eq:thm4eq1}). Then 
 \begin{multline}
   \label{eq:thm6eq1}
\lim_{n\rightarrow \infty } \sqrt{n} L^{(2)}   = 
   \frac{1}{\sqrt{\beta_1 \beta_2}} 
   \min_{c_3>0,\nu_1,\nu_2} 
   \Bigg (\Bigg .
  -\frac{c_{3}}{2} 
  + 
  \frac{\alpha}{c_3}\log \lp
    I^{(2)}(c_3,\nu_1;\beta_1)
        \rp + \nu_1\sqrt{\beta_1} 
       \\
        +
\frac{1}{c_3}\log \lp
       I^{(2)}(c_3,\nu_2;\beta_2)
      \rp + \nu_2\sqrt{\beta_2}         \Bigg .\Bigg ) .
\end{multline}
and
\begin{eqnarray}
   \label{eq:thm6eq2}
\lim_{n\rightarrow \infty}    \mE \xi\sqrt{n}   \leq  \lim_{n\rightarrow \infty}   \mE L^{(2)} \sqrt{n} \leq  \lim_{n\rightarrow \infty}   \mE L \sqrt{n}.
\end{eqnarray}
\end{theorem}

\begin{proof}
Follows from the preceding discussion.
\end{proof} 

The following corollary allows to obtain closed form solution for squared matrices.
 
\begin{corollary}
\label{cor:cor2}
 Assume the squared matrices setup of Theorem \ref{thm:thm6} with $\alpha=1$ and $\beta_1=\beta_2=\beta$. Set
\begin{eqnarray}
   \label{eq:cor2eq0}
b & = &  \frac{ 4 \left(1- \beta\right)  {\mathrm{e}}^{-\frac{\hat{\nu} ^2}{2}}}{\sqrt{2 \beta \pi }\, \mathrm{erfc}\left(\frac{-\,\hat{\nu} }{\sqrt{2}}\right)    } - 2\sqrt{\beta }\,\hat{\nu} 
\nonumber \\
\hat{c}_3 & = & \frac{-b + \sqrt{ b^2+12\log \lp  \frac{ \mathrm{erfc}\left(\frac{-\sqrt{2}\,\hat{\nu} }{2}\right)}{2(1- \beta )}
  \rp    } }{3}.
\end{eqnarray}
Let $\hat{\nu}$ satisfy
\begin{eqnarray}
   \label{eq:cor2eq0a0}
   {\mathrm{e}}^{-\frac{c_{3}\,\left(2\,\hat{\nu} -\frac{c_{3}}{\sqrt{\beta }}\right)}{2\,\sqrt{\beta }}}\,\mathrm{erfc}\left(\frac{\sqrt{2}\,\left(\hat{\nu} -\frac{c_{3}}{\sqrt{\beta }}\right)}{2}\right) 
& = &
\frac{\beta \, \mathrm{erfc}\left(\frac{-\sqrt{2}\,\hat{\nu} }{2}\right)}{1- \beta }
.
 \end{eqnarray}
 Then the optimal $\nu_1$ and $\nu_2$ in (\ref{eq:thm6eq1}) are equal to $\hat{\nu}$ and the optimal $c_3$ is equal to $\hat{c}_3$. Moreover,
 \begin{equation}
   \label{eq:cor2eq1}
\lim_{n\rightarrow \infty } \sqrt{n} L^{(2)}   = 
   \frac{1}{\beta} 
    \lp
 \hat{c}_3   + b    +2\hat{\nu}\sqrt{\beta}   \rp .
\end{equation}
and
\begin{eqnarray}
   \label{eq:cor2eq2}
\lim_{n\rightarrow \infty}    \mE \xi\sqrt{n}   \leq  \lim_{n\rightarrow \infty}   \mE L^{(2)} \sqrt{n} \leq  \lim_{n\rightarrow \infty}   \mE L \sqrt{n}
=
     \sqrt{\frac{2}{\pi\beta^3}}  e^{-\lp \erfinv(1-2\beta)\rp^2} 
.
\end{eqnarray}
  \end{corollary}

\begin{proof}
One first observes that in the square case $\nu_1=\nu_2$. Then (\ref{eq:thm6eq1}) becomes 
 \begin{equation}
   \label{eq:prcor2eq1}
\lim_{n\rightarrow \infty } \sqrt{n} L^{(2)}   = 
   \frac{1}{\beta} 
   \min_{c_3>0,\nu } 
   \Bigg (\Bigg .
  -\frac{c_{3}}{2} 
  + 
  \frac{2}{c_3}\log \lp
    I^{(2)}(c_3,\nu;\beta)
        \rp + 2\nu\sqrt{\beta} 
   \Bigg .\Bigg ) 
   \triangleq
   \frac{1}{\beta} 
   \min_{c_3>0,\nu } f^{(2)}(c_3,\nu;\beta)
   .
\end{equation}
Taking the $\nu$ derivative gives
\begin{eqnarray}
   \label{eq:prcor2eq2}
\frac{dI^{(2)}(c_3,\nu;\beta)}{d\nu}
=
 -\frac{c_{3}\,{\mathrm{e}}^{-\frac{c_{3}\,\left(2\,\nu -\frac{c_{3}}{\sqrt{\beta }}\right)}{2\,\sqrt{\beta }}}\,\mathrm{erfc}\left(\frac{\sqrt{2}\,\left(\nu -\frac{c_{3}}{\sqrt{\beta }}\right)}{2}\right)}{2\,\sqrt{\beta }},
\end{eqnarray}
and
\begin{eqnarray}
   \label{eq:prcor2eq3}
\frac{df^{(2)}(c_3,\nu;\beta)}{d\nu}
& = &  \frac{2}{c_3 I^{(2)}(c_3,\nu;\beta)} \frac{dI^{(2)}(c_3,\nu;\beta)}{d\nu}
+2\sqrt{\beta}.
\end{eqnarray}
Equalling the above derivative to zero and combining with (\ref{eq:hlrd6}) and (\ref{eq:prcor2eq3}), we find
 \begin{eqnarray}
   \label{eq:prcor2eq4}
\frac{c_{3}\,{\mathrm{e}}^{-\frac{c_{3}\,\left(2\,\nu -\frac{c_{3}}{\sqrt{\beta }}\right)}{2\,\sqrt{\beta }}}\,\mathrm{erfc}\left(\frac{\sqrt{2}\,\left(\nu -\frac{c_{3}}{\sqrt{\beta }}\right)}{2}\right)}{2\,\sqrt{\beta }}
& = & \sqrt{\beta} c_3 I^{(2)}(c_3,\nu;\beta)
\nonumber \\
& = & \sqrt{\beta} c_3  
  \lp   \frac{\mathrm{erfc}\left(-\frac{\sqrt{2}\,\nu }{2}\right)}{2}
      +\frac{{\mathrm{e}}^{-\frac{c_{3}\,\left(2\,\nu -\frac{c_{3}}{\sqrt{\beta }}\right)}{2\,\sqrt{\beta }}}\,\mathrm{erfc}\left(\frac{\sqrt{2}\,\left(\nu -\frac{c_{3}}{\sqrt{\beta }}\right)}{2}\right)}{2} \rp , \nonumber \\
 \end{eqnarray}
and
 \begin{eqnarray}
   \label{eq:prcor2eq5}
 {\mathrm{e}}^{-\frac{c_{3}\,\left(2\,\nu -\frac{c_{3}}{\sqrt{\beta }}\right)}{2\,\sqrt{\beta }}}\,\mathrm{erfc}\left(\frac{\sqrt{2}\,\left(\nu -\frac{c_{3}}{\sqrt{\beta }}\right)}{2}\right) 
 & = &    
     \frac{\beta \, \mathrm{erfc}\left(-\frac{\sqrt{2}\,\nu }{2}\right)}{1-\beta}
   .
 \end{eqnarray}
Moreover
 \begin{eqnarray}
   \label{eq:prcor2eq5a0}
I^{(2)}(c_3,\nu;\beta)
  & = &    
     \frac{ \mathrm{erfc}\left(-\frac{\sqrt{2}\,\nu }{2}\right)}{2(1-\beta)}
   .
 \end{eqnarray}
Taking $c_3$ derivative gives
 \begin{eqnarray}
   \label{eq:prcor2eq6}
   \frac{dI^{(2)}(c_3,\nu;\beta)}{dc_3}
 & = &
 \frac{{\mathrm{e}}^{-\frac{c_{3}\,\left(2\,\nu -\frac{c_{3}}{\sqrt{\beta }}\right)}{2\,\sqrt{\beta }}}\,\mathrm{erfc}\left(\frac{\sqrt{2}\,\left(\nu -\frac{c_{3}}{\sqrt{\beta }}\right)}{2}\right)\,\left(\frac{c_{3}}{\beta }-\frac{\nu }{\sqrt{\beta }}\right)}{2}
+\frac{ e^{-\frac{\nu^2}{2}} }{\sqrt{2 \beta \pi }  } 
\nonumber \\
 & = &
 \frac{      \frac{\beta \, \mathrm{erfc}\left(-\frac{\sqrt{2}\,\nu }{2}\right)}{1-\beta} \left(\frac{c_{3}}{\beta }-\frac{\nu }{\sqrt{\beta }}\right)}{2}
+\frac{ e^{-\frac{\nu^2}{2}} }{\sqrt{2 \beta \pi }  } 
\nonumber \\
 & = &
I^{(2)}(c_3,\nu;\beta) \left( c_{3}- \nu \sqrt{\beta }\right) 
+\frac{ e^{-\frac{\nu^2}{2}} }{\sqrt{2 \beta \pi }  } 
,
\end{eqnarray}
and
 \begin{eqnarray}
   \label{eq:prcor2eq7}
   \frac{df^{(2)}(c_3,\nu;\beta)}{dc_3}
 & = &
  -\frac{1}{2} 
  - 
  \frac{2}{c_3^2}\log \lp
    I^{(2)}(c_3,\nu;\beta)
        \rp +   
  \frac{2}{c_3 I^{(2)}(c_3,\nu;\beta)} \frac{d
    I^{(2)}(c_3,\nu;\beta) }{dc_3}
\nonumber \\
& = &
  -\frac{1}{2} 
  - 
  \frac{2}{c_3^2}\log \lp
    I^{(2)}(c_3,\nu;\beta)
        \rp +   
  \frac{2}{c_3 I^{(2)}(c_3,\nu;\beta)} 
   \lp
   I^{(2)}(c_3,\nu;\beta) \left( c_{3}- \nu \sqrt{\beta }\right) 
+\frac{ e^{-\frac{\nu^2}{2}} }{\sqrt{2 \beta \pi }  } 
   \rp
\nonumber \\
& = &
  \frac{3}{2} 
  - 
  \frac{2}{c_3^2}\log \lp
    I^{(2)}(c_3,\nu;\beta)
        \rp +   
  \frac{2}{c_3 } 
   \lp
  - \nu \sqrt{\beta }  
+\frac{ e^{-\frac{\nu^2}{2}} }{\sqrt{2 \beta \pi }  I^{(2)}(c_3,\nu;\beta) } 
   \rp
\nonumber \\
& = &
  \frac{3}{2} 
  - 
  \frac{2}{c_3^2}\log \lp
    I^{(2)}(c_3,\nu;\beta)
        \rp +   
  \frac{1}{c_3 } 
   \lp
  - 2\nu \sqrt{\beta }  
+\frac{4(1-\beta) e^{-\frac{\nu^2}{2}} }{\sqrt{2 \beta \pi } \, \mathrm{erfc}\left(-\frac{\sqrt{2}\,\nu }{2}\right) } 
   \rp
\nonumber \\
& = &
  \frac{3}{2} 
  - 
  \frac{2}{c_3^2}\log \lp
 I^{(2)}(c_3,\nu;\beta)        \rp +   
  \frac{1}{c_3 } 
b
\nonumber \\
& = &
  \frac{3}{2} 
  - 
  \frac{2}{c_3^2}\log \lp
     \frac{ \mathrm{erfc}\left(-\frac{\sqrt{2}\,\nu }{2}\right)}{2(1-\beta)}
        \rp +   
  \frac{1}{c_3 } 
b .
  \end{eqnarray}
Equalling the above derivative to zero and solving over $c_3$ gives
\begin{eqnarray}
   \label{eq:prcor2eq8}
c_3 & = & \frac{-b + \sqrt{ b^2+12\log \lp  \frac{ \mathrm{erfc}\left(\frac{-\sqrt{2}\,\nu }{2}\right)}{2(1- \beta )}
  \rp    } }{3}.
\end{eqnarray}
Utilizing (\ref{eq:prcor2eq7}), we then also observe 
\begin{eqnarray}
   \label{eq:prcor2eq9}
  f^{(2)}(c_3,\nu;\beta) =  
  -\frac{c_{3}}{2} 
  + 
  \frac{2}{c_3}\log \lp
    I^{(2)}(c_3,\nu;\beta)
        \rp + 2\nu\sqrt{\beta}  
=  
  -\frac{c_{3}}{2} 
  + \frac{3c_{3}}{2} +b
 + 2\nu\sqrt{\beta}
 =  
c_3 +b
 + 2\nu\sqrt{\beta}.
 .  
\end{eqnarray}
Comparing (\ref{eq:prcor2eq8}), (\ref{eq:prcor2eq5}), and  (\ref{eq:prcor2eq9}) to 
(\ref{eq:cor2eq0}), (\ref{eq:cor2eq0a0}), and (\ref{eq:cor2eq1}), while keeping in mind (\ref{eq:prcor2eq1}), completes the proof.
\end{proof}

We also have the following $\beta\rightarrow 0$ limiting corollary.

\begin{corollary}
\label{cor:cor3}
As in Corollary \ref{cor:cor2},  assume the squared matrices setup of Theorem \ref{thm:thm6} with $\alpha=1$ and $\beta_1=\beta_2=\beta$. Additionally, let $\beta\rightarrow 0$. Then
\begin{equation}
   \label{eq:cor3eq1}
\lim_{n\rightarrow \infty } \sqrt{n} L^{(2)}   \longrightarrow  
 2\sqrt{\frac{1}{\beta} \log\lp \frac{1}{\beta}\rp  },
\end{equation}
and
\begin{eqnarray}
   \label{eq:cor3eq2}
\lim_{n\rightarrow \infty}    \mE \xi\sqrt{n}   \leq  \lim_{n\rightarrow \infty}   \mE L^{(2)} \sqrt{n} \longrightarrow  
 2\sqrt{\frac{1}{\beta} \log\lp \frac{1}{\beta}\rp  }.
\end{eqnarray}
  \end{corollary}

\begin{proof}
We start by assuming that $b$ and $\hat{c}_3$ in (\ref{eq:cor2eq0}) behaves as
\begin{eqnarray}
   \label{eq:prfcor3eq1}
b & = &  \frac{ 4 \left(1- \beta\right)  {\mathrm{e}}^{-\frac{\hat{\nu} ^2}{2}}}{\sqrt{2 \beta \pi }\, \mathrm{erfc}\left(\frac{-\,\hat{\nu} }{\sqrt{2}}\right)    } - 2\sqrt{\beta }\,\hat{\nu} 
\longrightarrow 
- 2\sqrt{\beta }\,\hat{\nu} 
\nonumber \\
\hat{c}_3 & = & \frac{-b + \sqrt{ b^2+12\log \lp  \frac{ \mathrm{erfc}\left(\frac{-\sqrt{2}\,\hat{\nu} }{2}\right)}{2(1- \beta )}
  \rp    } }{3}
  \longrightarrow 
  -\frac{2b}{3} \longrightarrow   \frac{4}{3}\sqrt{\beta }\,\hat{\nu} 
.
\end{eqnarray}
Assuming $\hat{\nu}\longrightarrow \infty$, we observe
\begin{eqnarray}
   \label{eq:prfcor3eq2}
   {\mathrm{e}}^{-\frac{c_{3}\,\left(2\,\hat{\nu} -\frac{c_{3}}{\sqrt{\beta }}\right)}{2\,\sqrt{\beta }}}\,\mathrm{erfc}\left(\frac{\sqrt{2}\,\left(\hat{\nu} -\frac{c_{3}}{\sqrt{\beta }}\right)}{2}\right) 
&   \longrightarrow &
   {\mathrm{e}}^{- \frac{4}{9}\hat{\nu}^2 } \,\mathrm{erfc}\left(\frac{-\sqrt{2}\,  \hat{\nu} }{6}\right)    \longrightarrow
  2 {\mathrm{e}}^{- \frac{4}{9}\hat{\nu}^2 } 
\nonumber \\
 \frac{\beta \, \mathrm{erfc}\left(\frac{-\sqrt{2}\,\hat{\nu} }{2}\right)}{1- \beta }
&   \longrightarrow &
2\beta.
 \end{eqnarray}
A combination of  (\ref{eq:cor2eq0a0}) and (\ref{eq:prfcor3eq1})  gives
\begin{eqnarray}
   \label{eq:prfcor3eq3}
  2 {\mathrm{e}}^{- \frac{4}{9}\hat{\nu}^2 } 
 \longrightarrow 2\beta \quad \mbox{and} \quad
 \hat{\nu} \longrightarrow \frac{3}{2} \sqrt{\log\lp \frac{1}{\beta} \rp} \longrightarrow \infty,
 \end{eqnarray}
 which confirms $\hat{\nu} \longrightarrow \infty$ assumption. For $\hat{\nu}$ from (\ref{eq:prfcor3eq3}), $b$ in  (\ref{eq:prfcor3eq3a0}) can be rewritten as
  \begin{eqnarray}
   \label{eq:prfcor3eq3a0}
b & = &  
\frac{ 4 \left(1- \beta\right)  {\mathrm{e}}^{-\frac{\hat{\nu} ^2}{2}}}{\sqrt{2 \beta \pi }\, \mathrm{erfc}\left(\frac{-\,\hat{\nu} }{\sqrt{2}}\right)    } - 2\sqrt{\beta }\,\hat{\nu} 
\longrightarrow 
\frac{ 2 \left(1- \beta\right)  {\mathrm{e}}^{\frac{9}{8} \log\lp \beta \rp  } }  {\sqrt{2 \beta \pi }   } - 2\sqrt{\beta }\,\hat{\nu} 
\longrightarrow 
\frac{ 2 \left(1- \beta\right)  \beta^{\frac{5}{8}}  }  {\sqrt{2 \pi }   } - 2\sqrt{\beta }\,\hat{\nu} 
\longrightarrow 
- 2\sqrt{\beta }\,\hat{\nu} 
 .\nonumber \\
\end{eqnarray}
Similarly $\hat{c}_3$ in  (\ref{eq:prfcor3eq3a0}) can be rewritten as
  \begin{eqnarray}
   \label{eq:prfcor3eq3a1}
 \hat{c}_3 & = & \frac{-b + \sqrt{ b^2+12\log \lp  \frac{ \mathrm{erfc}\left(\frac{-\sqrt{2}\,\hat{\nu} }{2}\right)}{2(1- \beta )}
  \rp    } }{3}
  \longrightarrow 
 \frac{-b + \sqrt{ b^2+12\log \lp  1 - \frac{e^{-\frac{\hat{\nu}^2}{2} }}{\sqrt{2\pi} \hat{\nu}}  
  \rp    } }{3}
\nonumber \\
&  \longrightarrow &
 \frac{-b + \sqrt{ b^2 - 12\frac{e^{-\frac{\hat{\nu}^2}{2} }}{\sqrt{2\pi} \hat{\nu}}  
     } }{3}
  \longrightarrow 
 \frac{-b + \sqrt{ b^2 - 12\frac{\beta^{\frac{9}{8}}}{\sqrt{2\pi} \hat{\nu}}  
     } }{3}
 \longrightarrow 
 \frac{-b + |b| \sqrt{ 1 - 3\frac{\beta^{\frac{1}{8}}}{\sqrt{2\pi} \hat{\nu}^3}  
     } }{3}
\nonumber \\
&   \longrightarrow  &
   -\frac{2b}{3} \longrightarrow   \frac{4}{3}\sqrt{\beta }\,\hat{\nu} .
\end{eqnarray}
Together, (\ref{eq:prfcor3eq3a0}) and    (\ref{eq:prfcor3eq3a1}) confirm  that the assumption made in (\ref{eq:prfcor3eq1}) is indeed correct. From (\ref{eq:cor2eq1}), (\ref{eq:prfcor3eq1}), and  (\ref{eq:prfcor3eq3}), we then find
\begin{equation}
   \label{eq:prfcor2eq4}
\lim_{n\rightarrow \infty } \sqrt{n} L^{(2)}   = 
   \frac{1}{\beta} 
    \lp
 \hat{c}_3   + b    +2\hat{\nu}\sqrt{\beta}   \rp 
 \longrightarrow \frac{4}{3\sqrt{\beta}} \hat{\nu}
 \longrightarrow  2 \sqrt{  \frac{1}{\beta} \log\lp \frac{1}{\beta} \rp}
 ,
\end{equation}
which matches (\ref{eq:cor3eq1}) and completes the proof. 
\end{proof}

Comparing to Corollary \ref{cor:cor1a0}, we observe that the lifting effect is significant and the plain RDT bound is lowered precisely $\sqrt{2}$ times for $\beta\rightarrow 0$.

We also note that the results obtained in Corollary \ref{cor:cor3} precisely match the ones obtained via replica methods in \cite{ErbaKOZ24} assuming one step of replica symmetry breaking. Moreover, recognizing
\begin{equation}\label{eq:mcteq1}
  2 \sqrt{  \frac{1}{\beta} \log\lp \frac{1}{\beta} \rp}
  \longrightarrow
  2 \sqrt{  \frac{n}{k} \log\lp \frac{n}{k} \rp}
  =
  2 \sqrt{n} \sqrt{  \frac{1}{k}\lp \log (n) -\log(k)  \rp } .
\end{equation}
and assuming that $k$ is large (either as completely independent of $n$ or as a sublinear function of $n$), one also has
\begin{equation}\label{eq:mcteq2}
   2 \sqrt{n} \sqrt{  \frac{1}{k} \lp \log (n) -\log(k)  \rp } 
   \longrightarrow 
      2 \sqrt{n} \sqrt{  \frac{\log (n)}{k}  } .
\end{equation}
The term on the right-hand side is precisely what was obtained in sublinear regime in \cite{BhamidiDN17,GamarnikLi18}.

   \begin{figure}[h]
%\begin{minipage}[b]{.5\linewidth}
\centering
\centerline{\includegraphics[width=.87\linewidth]{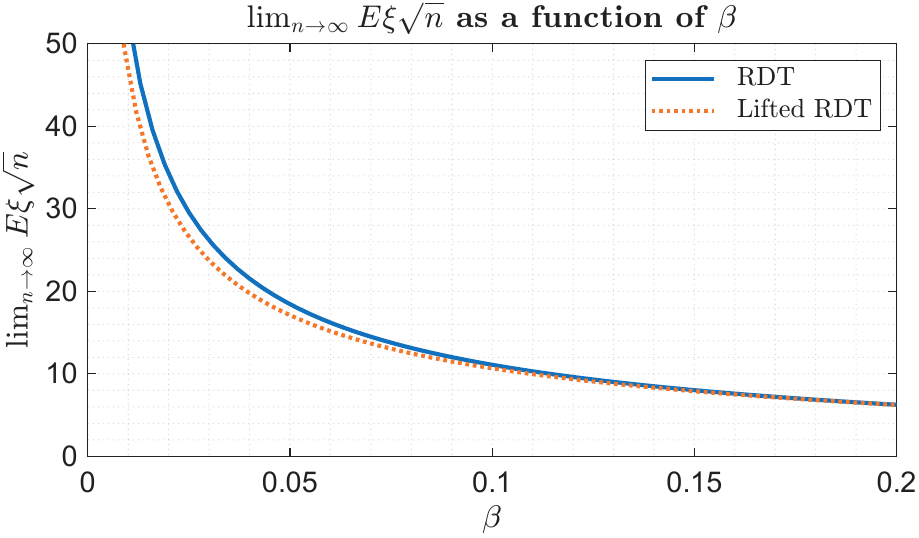}}
%\end{minipage}
%\begin{minipage}[b]{.5\linewidth}
%\centering
%\centerline{\epsfig{figure=finprerral08.eps,width=9cm,height=6.5cm}}
%\end{minipage}
\caption{ RDT and Lifted RDT upper bounds on the Largest average submatrix value $\lim_{n\rightarrow \infty} \mE \xi\sqrt{n}$  as functions of $\beta$.}
\label{fig:fig1}
\end{figure}

The results established in Corollaries \ref{cor:cor1} and \ref{cor:cor2} are sufficient to obtain concrete upper bounds on the largest average submatrix values, as shown in Figure \ref{fig:fig1}. As demonstrated, the lifted RDT improves upon the plain RDT, confirming that the strong RDT is not in place. We have particularly highlighted the range $\beta\in(0,0.2)$, as the lifting effect begins to take place at $\beta\approx 0.24$.

   \begin{figure}[h]
%\begin{minipage}[b]{.5\linewidth}
\centering
\centerline{\includegraphics[width=.87\linewidth]{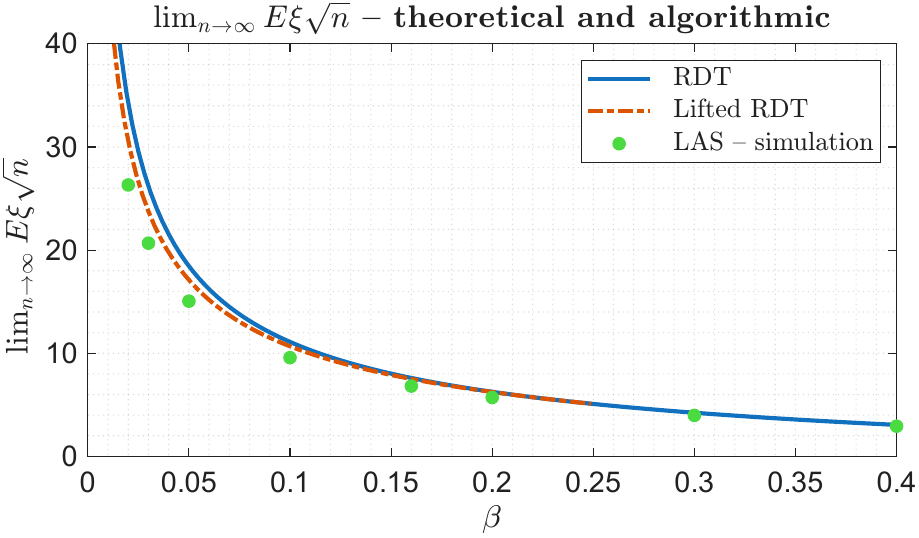}}
%\end{minipage}
%\begin{minipage}[b]{.5\linewidth}
%\centering
%\centerline{\epsfig{figure=finprerral08.eps,width=9cm,height=6.5cm}}
%\end{minipage}
\caption{ Largest average submatrix value $\lim_{n\rightarrow \infty} \mE \xi\sqrt{n}$  -- theoretical (RDT and Lifted RDT) versus algorithmic (LAS).}
\label{fig:fig2}
\end{figure}

We implemented a simple Largest Average Submatrix (LAS) algorithm \cite{ShabalinWPN09}. Its results are shown in Figure \ref{fig:fig2}. Despite being a suboptimal heuristic, LAS fairly closely approaches  the theoretical RDT predictions. Additionally, we chose a slightly wider range for $\beta$ to demonstrate that the algorithm performs even better and more closely approaches the theoretical upper bounds in the higher $\beta$ regime.

Figures \ref{fig:fig1} and  \ref{fig:fig2} are supplemented with Tables \ref{tab:tab1} and \ref{tab:tab2} where concrete values of all relevant parameters are provided. Table \ref{tab:tab1} relates to larger values of $\beta$ whereas Table \ref{tab:tab2} relates to smaller values of $\beta$ where lifting effect is more pronounced. As can be seen, for $\beta\geq 0.7$ the LAS algorithmic values are within $0.99$ of the RDT theoretical estimates indicating that the statistical-computational gap is either completely absent or its presence is for all practical purposes irrelevant. 

\begin{table}[h]
\caption{LASP -- Theory (RDT) versus Simulation (LAS) ; $m=n=3000$; Higher $\beta$ regime ($ \beta=\lim_{n\rightarrow \infty}\frac{k}{n}$) }\vspace{.1in}
%\begin{adjustwidth}{-1.4cm}{}
\centering
\def\arraystretch{1.2}
\begin{tabular}{||c||c|c||c|c||c|c||}\hline\hline
$\beta$  &   \multicolumn{2}{c||}{$ 0.5 $} &  \multicolumn{2}{c||}{$ 0.7 $} &  \multicolumn{2}{c||}{$ 0.88 $}   \\
  \hline\hline 
Theory vs Simulation  & \multicolumn{1}{c|}{\textbf{Theory}}
& \multicolumn{1}{c||}{\bl{\textbf{Sim.}}}
  & \multicolumn{1}{c|}{\textbf{Theory}}
& \multicolumn{1}{c||}{\bl{\textbf{Sim.}}}
  & \multicolumn{1}{c|}{\textbf{Theory}}
& \multicolumn{1}{c||}{\bl{\textbf{Sim.}}}
   \\
  \hline 
Methodology   &  RDT &    LAS 
&   RDT &     LAS  
&   RDT &     LAS  
 \\
\hline\hline
$\nu$ & $0 $ & $ -- $ 
 & $-0.524$  & $ -- $  & $-1.175$ & $ -- $  \\
 \hline\hline
 $\lim_{n\rightarrow \infty} \mE \xi \sqrt{n} $    & $\mathbf{2.25}$  & \bl{$\mathbf{2.20}$}  
   & $\mathbf{1.187}$ & \bl{$\mathbf{1.174}$} & $\mathbf{0.484}$ & \bl{$\mathbf{0.483}$} 
   \\ \hline\hline
  \end{tabular}
%\end{adjustwidth}
\label{tab:tab1}
\end{table}

\begin{table}[h]
\caption{LASP -- Theory (RDT, Lifted RDT) versus Simulation (LAS) ; $m=n=3000$; Lower $\beta$ regime ($ \beta=\lim_{n\rightarrow \infty}\frac{k}{n}$) }\vspace{.1in}
%\begin{adjustwidth}{-1.4cm}{}
\centering
\def\arraystretch{1.2}
\begin{tabular}{||c||c|c|c||c|c|c||}\hline\hline
$\beta$  &   \multicolumn{3}{c||}{$ 0.1 $} &  \multicolumn{3}{c||}{$ 0.2 $}   \\
  \hline\hline 
Theory vs Simulation  & \multicolumn{2}{c|}{\textbf{Theory}}
& \multicolumn{1}{c||}{\bl{\textbf{Sim.}}}
  & \multicolumn{2}{c|}{\textbf{Theory}}
& \multicolumn{1}{c||}{\bl{\textbf{Sim.}}}
   \\
  \hline 
Methodology   &  RDT &   Lifted RDT &   LAS &  
 RDT &   Lifted RDT &   LAS  
 \\
\hline\hline
$\nu$ & $1.282$ & $1.810$ & $ -- $ 
 & $0.842$ & $1.028$ & $ -- $  \\
\hline
$c_3$ & $ \rightarrow 0$  & $0.610$ & $ -- $ 
& $ \rightarrow 0$  & $0.255$ & $ -- $ 
\\
\hline\hline
 $\lim_{n\rightarrow \infty} \mE \xi \sqrt{n} $    & $\mathbf{11.099}$  & $\mathbf{10.677}$ & \bl{$\mathbf{9.57}$}  
   & $\mathbf{6.260}$  & $\mathbf{6.239}$ & \bl{$\mathbf{5.71}$} 
   \\ \hline\hline
  \end{tabular}
%\end{adjustwidth}
\label{tab:tab2}
\end{table}

%%%%%%%%%%%%%%%%%%%%%%%%%%%%%%%%%%%%%%%%%%%%%%%%%%%%%%%%%%%%%%%%%%%%%%%%%%%%%%%%
%%%%%%%%%%%%%%%%%%%%%%%%%%%%%%%%%%%%%%%%%%%%%%%%%%%%%%%%%%%%%%%%%%%%%%%%%%%%%%%%
%%%%%%%%%%%%%%%%%%%%%%%%%%%%%%%%%%%%%%%%%%%%%%%%%%%%%%%%%%%%%%%%%%%%%%%%%%%%%%%%
%%%%%%%%%%%%%%%%%%%%%%%%%%%%%%%%%%%%%%%%%%%%%%%%%%%%%%%%%%%%%%%%%%%%%%%%%%%%%%%%
%%%%%%%%%%%%%%%%%%%%%%%%%%%%%%%%%%%%%%%%%%%%%%%%%%%%%%%%%%%%%%%%%%%%%%%%%%%%%%%%
\section{Conclusion}
\label{sec:conc}
%%%%%%%%%%%%%%%%%%%%%%%%%%%%%%%%%%%%%%%%%%%%%%%%%%%%%%%%%%%%%%%%%%%%%%%%%%%%%%%%
%%%%%%%%%%%%%%%%%%%%%%%%%%%%%%%%%%%%%%%%%%%%%%%%%%%%%%%%%%%%%%%%%%%%%%%%%%%%%%%%
%%%%%%%%%%%%%%%%%%%%%%%%%%%%%%%%%%%%%%%%%%%%%%%%%%%%%%%%%%%%%%%%%%%%%%%%%%%%%%%%
%%%%%%%%%%%%%%%%%%%%%%%%%%%%%%%%%%%%%%%%%%%%%%%%%%%%%%%%%%%%%%%%%%%%%%%%%%%%%%%%
%%%%%%%%%%%%%%%%%%%%%%%%%%%%%%%%%%%%%%%%%%%%%%%%%%%%%%%%%%%%%%%%%%%%%%%%%%%%%%%%

We have studied the largest average submatrix problem (LASP), focusing on the linear regime where submatrix dimensions are linear functions of the original matrix dimensions. While sublinear regimes have been thoroughly studied, the linear regime has lacked mathematically rigorous results supporting the presence or absence of a statistical-computational gap (SCG).

To address this, we developed a framework based on Random Duality Theory (RDT) to study LASP. Using both plain and lifted RDT variants, we obtained closed-form upper bounds on the LASP objectives. We found that lifted RDT strictly improves upon plain RDT within a certain range of submatrix dimensions. Furthermore, for small submatrices in the linear regime, we proved that our results match the replica predictions from \cite{ErbaKOZ24} obtained with one-step replica symmetry breaking ansatz (and their extrapolated sublinear analogues match \cite{BhamidiDN17,GamarnikLi18}).

Additionally, implementation of a simple LAS algorithm  \cite{ShabalinWPN09} uncovered that its performance closely approaches our theoretical predictions already for matrix sizes on the order of one thousand. For a significant portion of the submatrix dimension range, the theoretical bounds and algorithmic performance differ by no more than $~0.1\%$. This suggests that (at least in such dimensions range) the SCG is either completely absent or practically insignificant.

%%%%%%%%%%%%%%%%%%%%%%%%%%%%%%%%%%%%%%%%%%%%%%%%%%%%%%%%%%%%%%%%%%%%%%%%%%%%%%%%
%%%%%%%%%%%%%%%%%%%%%%%%%%%%%%%%%%%%%%%%%%%%%%%%%%%%%%%%%%%%%%%%%%%%%%%%%%%%%%%%
%\section*{Acknowledgment}
%%%%%%%%%%%%%%%%%%%%%%%%%%%%%%%%%%%%%%%%%%%%%%%%%%%%%%%%%%%%%%%%%%%%%%%%%%%%%%%%
%%%%%%%%%%%%%%%%%%%%%%%%%%%%%%%%%%%%%%%%%%%%%%%%%%%%%%%%%%%%%%%%%%%%%%%%%%%%%%%%

%%%%%%%%%%%%%%%%%%%%%%%%%%%%%%%%%%%%%%%%%%%%%%%%%%%%%%%%%%%%%%%%%%%%%%%%%%%%%%%%%
%%%%%%%%%%%%%%%%%%%%%%%%%%%%%%%%%%%%%%%%%%%%%%%%%%%%%%%%%%%%%%%%%%%%%%%%%%%%%%%%%
%\section*{Data availability}
%%%%%%%%%%%%%%%%%%%%%%%%%%%%%%%%%%%%%%%%%%%%%%%%%%%%%%%%%%%%%%%%%%%%%%%%%%%%%%%%%
%%%%%%%%%%%%%%%%%%%%%%%%%%%%%%%%%%%%%%%%%%%%%%%%%%%%%%%%%%%%%%%%%%%%%%%%%%%%%%%%%
%
%No data besides those given in the paper are used.
%
% 
%
%
%%%%%%%%%%%%%%%%%%%%%%%%%%%%%%%%%%%%%%%%%%%%%%%%%%%%%%%%%%%%%%%%%%%%%%%%%%%%%%%%%
%%%%%%%%%%%%%%%%%%%%%%%%%%%%%%%%%%%%%%%%%%%%%%%%%%%%%%%%%%%%%%%%%%%%%%%%%%%%%%%%%
%\section*{Competing interests}
%%%%%%%%%%%%%%%%%%%%%%%%%%%%%%%%%%%%%%%%%%%%%%%%%%%%%%%%%%%%%%%%%%%%%%%%%%%%%%%%%
%%%%%%%%%%%%%%%%%%%%%%%%%%%%%%%%%%%%%%%%%%%%%%%%%%%%%%%%%%%%%%%%%%%%%%%%%%%%%%%%%
%
%The authors have no competing interests to declare that are relevant to the content of this article.
%

%\newpage1
%\setcounter{page}{1}
\begin{singlespace}
\bibliographystyle{plain}
\bibliography{nflgscompyxRefs}

\begin{thebibliography}{100}

\bibitem{AbbLiSly21b}
E.~Abbe, S.~Li, and A.~Sly.
\newblock Proof of the contiguity conjecture and lognormal limit for the
  symmetric perceptron.
\newblock In {\em 62nd {IEEE} Annual Symposium on Foundations of Computer
  Science, {FOCS} 2021, Denver, CO, USA, February 7-10, 2022}, pages 327--338.
  {IEEE}, 2021.

\bibitem{AchlioptasR06}
D.~Achlioptas and F.~Ricci{-}Tersenghi.
\newblock On the solution-space geometry of random constraint satisfaction
  problems.
\newblock In {\em Proceedings of the 38th Annual {ACM} Symposium on Theory of
  Computing, Seattle, WA, USA, May 21-23, 2006}, pages 130--139. {ACM}, 2006.

\bibitem{ElAlGam24}
A.~El Alaoui and D.~Gamarnik.
\newblock Hardness of sampling solutions from the symmetric binary perceptron.
\newblock {\em Random Struct. Algorithms}, 66(4), 2025.

\bibitem{AlaouiMS21}
A.~El Alaoui, A.~Montanari, and M.~Sellke.
\newblock Optimization of mean-field spin glasses.
\newblock {\em The Annals of Probability}, 49(6), 2021.

\bibitem{AlaSel20}
A.~El Alaoui and M.~Sellke.
\newblock Algorithmic pure states for the negative spherical perceptron.
\newblock {\em Journal of Statistical Physics}, 189(27), 2022.

\bibitem{AMZ24}
B.~L. Annesi, E.~M. Malatesta, and F.~Zamponi.
\newblock {Exact full-RSB SAT/UNSAT transition in infinitely wide two-layer
  neural networks}.
\newblock {\em SciPost Phys.}, 18:118, 2025.

\bibitem{AubPerZde19}
B.~Aubin, W.~Perkins, and L.~Zdeborova.
\newblock Storage capacity in symmetric binary perceptrons.
\newblock {\em J. Phys. A}, 52(29):294003, 2019.

\bibitem{BBP05}
J.~Baik, G.~Ben~Arous, and S.~Peche.
\newblock Phase transition of the largest eigenvalue for non-null complex
  sample covariance matrices.
\newblock {\em The Annals of Probability}, 33(5):1643--1697, 2005.

\bibitem{Bald15}
C.~Baldassi, A.~Ingrosso, C.~Lucibello, L.~Saglietti, and R.~Zecchina.
\newblock Subdominant dense clusters allow for simple learning and high
  computational performance in neural networks with discrete synapses.
\newblock {\em Physical Review letters}, 115(12):128101, 2015.

\bibitem{BMPZ23}
C.~Baldassi, E.~M. Malatesta, G.~Perugini, and R.~Zecchina.
\newblock Typical and atypical solutions in nonconvex neural networks with
  discrete and continuous weights.
\newblock {\em Phys. Rev. E}, 108:024310, Aug 2023.

\bibitem{Bald20}
C.~Baldassi, R.~D. Vecchia, C.~Lucibello, and R.~Zecchina.
\newblock Clustering of solutions in the symmetric binary perceptron.
\newblock {\em Journal of Statistical Mechanics: Theory and Experiment},
  (7):073303, 2020.

\bibitem{BanksMVVX18}
J.~Banks, C.~Moore, R.~Vershynin, N.~Verzelen, and J.~Xu.
\newblock Information-theoretic bounds and phase transitions in clustering,
  sparse pca, and submatrix localization.
\newblock {\em IEEE Transactions on Information Theory}, 64(7):4872--4894,
  2018.

\bibitem{BanSpen20}
N.~Bansal and J.~H Spencer.
\newblock On-line balancing of random inputs.
\newblock {\em Random Structures \& Algorithms}, 57(4):879--891, 2020.

\bibitem{BarbAKZ23}
D.~Barbier, A.~El Alaoui, F.~Krzakala, and L.~Zdeborova.
\newblock On the atypical solutions of the symmetric binary perceptron.
\newblock {\em Journal of Physics A: Mathematical and Theoretical},
  57(19):195202, 2024.

\bibitem{BarbierMKLZ16}
J.~Barbier, M.~Dia, N.~Macris, F.~Krzakala, T.~Lesieur, and L.~Zdeborova.
\newblock Mutual information for symmetric rank-one matrix estimation: A proof
  of the replica formula.
\newblock In {\em Advances in Neural Information Processing Systems (NIPS
  2016)}, pages 424--432, 2016.

\bibitem{BarbierKMMZ18}
J.~Barbier, F.~Krzakala, N.~Macris, L.~Miolane, and L.~Zdeborov{\'{a}}.
\newblock Optimal errors and phase transitions in high-dimensional generalized
  linear models.
\newblock In {\em Conference On Learning Theory, {COLT} 2018, Stockholm,
  Sweden, 6-9 July 2018}, volume~75 of {\em Proceedings of Machine Learning
  Research}, pages 728--731. {PMLR}, 2018.
\newblock available online at \bl{\url{http://arxiv.org/abs/1708.03395}}.

\bibitem{BehneReeves22}
J.~K. Behne and G.~Reeves.
\newblock Fundamental limits for rank-one matrix estimation with groupwise
  heteroskedasticity.
\newblock In {\em Proceedings of The 25th International Conference on
  Artificial Intelligence and Statistics}, volume 151 of {\em Proceedings of
  Machine Learning Research}, pages 8650--8672. PMLR, 2022.

\bibitem{BhamidiDN17}
S.~Bhamidi, P.~S. Dey, and A.~B. Nobel.
\newblock Energy landscape for large average submatrix detection problems in
  gaussian random matrices.
\newblock {\em Probability Theory and Related Fields}, 168(3-4):919--971, 2017.

\bibitem{BhaGG26}
S.~Bhamidi, D.~Gamarnik, and S.~Gong.
\newblock Finding a dense submatrix of a random matrix. sharp bounds for online
  algorithms.
\newblock {\em Electron. Commun. Probab.}, 31, 2026.

\bibitem{BoltNakSunXu22}
E.~Bolthausen, S.~Nakajima, N.~Sun, and C.~Xu.
\newblock Gardner formula for {I}sing perceptron models at small densities.
\newblock {\em {P}roceedings of {T}hirty {F}ifth {C}onference on {L}earning
  {T}heory, {PMLR}}, 178:1787--1911, 2022.

\bibitem{BrZech06}
A.~Braunstein and R.~Zecchina.
\newblock Learning by message passing in networks of discrete synapses.
\newblock {\em Physical review letters}, 96(3):030201, 2006.

\bibitem{BrennanBH19}
M.~Brennan, G.~Bresler, and W.~Huleihel.
\newblock Universality of computational lower bounds for submatrix detection.
\newblock In {\em Proceedings of Thirty Second Conference on Learning Theory},
  volume~99 of {\em Proceedings of Machine Learning Research}, pages 417--468,
  25--28 Jun 2019.

\bibitem{BreHuang21}
G.~Bresler and B.~Huang.
\newblock The algorithmic phase transition of random k-{SAT} for low degree
  polynomials.
\newblock In {\em 62th {IEEE} Annual Symposium on Foundations of Computer
  Science, {FOCS} 2021}, pages 298--309. {IEEE}, 2021.

\bibitem{ButIng13}
C.~Butucea and Y.~I. Ingster.
\newblock Detection of a sparse submatrix of a high-dimensional noisy matrix.
\newblock {\em Bernoulli}, 19(5B):2652--2688, 2013.

\bibitem{CaiLR17}
T.~T. Cai, T.~Liang, and A.~Rakhlin.
\newblock Computational and statistical boundaries for submatrix localization
  in a large noisy matrix.
\newblock {\em The Annals of Statistics}, 45(4):1403--1456, 2017.

\bibitem{CGPR19}
W.-K. Chen, D.~Gamarnik, D.~Panchenko, and M.~Rahman.
\newblock Suboptimality of local algorithms for a class of max-cut problems.
\newblock {\em The Annals of Probability}, 47(3):1587--1618, 2019.

\bibitem{Cover65}
T.~Cover.
\newblock Geomretrical and statistical properties of systems of linear
  inequalities with applications in pattern recognition.
\newblock {\em IEEE Transactions on Electronic Computers}, (EC-14):326--334,
  1965.

\bibitem{DadonHB24}
M.~Dadon, W.~Huleihel, and T.~Bendory.
\newblock Detection and recovery of hidden submatrices.
\newblock {\em {IEEE} Trans. Signal Inf. Process. over Networks}, 10:69--82,
  2024.

\bibitem{Dandietal25}
Y.~Dandi, D.~Gamarnik, F.~Pernice, and L.~Zdeborova.
\newblock Sequential dynamics in {I}sing spin glasses.
\newblock 2025.
\newblock available online at \bl{\url{http://arxiv.org/abs/2506.09877}}.

\bibitem{DeshMont14}
Y.~Deshpande and A.~Montanari.
\newblock Information-theoretically optimal sparse pca.
\newblock In {\em 2014 IEEE International Symposium on Information Theory},
  pages 2197--2201. IEEE, 2014.

\bibitem{DinSlySun15}
J.~Ding, A.~Sly, and N.~Sun.
\newblock Satisfiability {T}hreshold for {R}andom {R}egular {NAE-SAT}.
\newblock {\em Communications in Mathematical Physics}, 341(2):435--489, 2015.

\bibitem{DingSun19}
J.~Ding and N.~Sun.
\newblock Capacity lower bound for the {I}sing perceptron.
\newblock {\em {STOC} 2019: Proceedings of the 51st Annual {ACM SIGACT}
  {S}ymposium on {T}heory of {C}omputing}, pages 816--827, 2019.

\bibitem{DonGavJohn18}
D.~L. Donoho, M.~Gavish, and I.~M. Johnstone.
\newblock Optimal shrinkage of eigenvalues in the spiked covariance model.
\newblock {\em The Annals of Statistics}, 46(4):1742--1778, 2018.

\bibitem{ErbaKOZ24}
V.~Erba, F.~Krzakala, R.~P. Ortiz, and L.~Zdeborova.
\newblock Statistical mechanics of the maximum-average submatrix problem.
\newblock {\em Journal of Statistical Mechanics: Theory and Experiment},
  1:013403, 2024.

\bibitem{ErbaKOZ26}
V.~Erba, N.~M. Kupferschmid, R.~P. Ortiz, and L.~Zdeborova.
\newblock The maximum-average subtensor problem: equilibrium and
  out-of-equilibrium properties.
\newblock {\em SciPost Phys.}, 20:073, 2026.

\bibitem{Fort10}
S.~Fortunato.
\newblock Community detection in graphs.
\newblock {\em Physics Reports}, 486:75--174, 2010.

\bibitem{FraPar16}
S.~Franz and G.~Parisi.
\newblock The simplest model of jamming.
\newblock {\em Journal of Physics A: Mathematical and Theoretical},
  49(14):145001, 2016.

\bibitem{FPSUZ17}
S.~Franz, G.~Parisi, M.~Sevelev, P.~Urbani, and F.~Zamponi.
\newblock Universality of the {SAT-UNSAT} (jamming) threshold in non-convex
  continuous constraint satisfaction problems.
\newblock {\em SciPost Physics}, 2:019, 2017.

\bibitem{Gamar21}
D.~Gamarnik.
\newblock The overlap gap property: A topological barrier to optimizing over
  random structures.
\newblock {\em Proceedings of the National Academy of Sciences}, 118(41), 2021.

\bibitem{GamJag21}
D.~Gamarnik and A.~Jagannath.
\newblock The overlap gap property and approximate message passing algorithms
  for p-spin models.
\newblock {\em Ann. Probab.}, 49:180 -- 205, 2021.

\bibitem{GamarnikJS21}
D.~Gamarnik, A.~Jagannath, and S.~Sen.
\newblock The overlap gap property in principal submatrix recovery.
\newblock {\em Probability Theory and Related Fields}, 181(4):797--837, 2021.

\bibitem{GamAW24}
D.~Gamarnik, A.~Jagannath, and A.~S. Wein.
\newblock Hardness of random optimization problems for {B}oolean circuits,
  low-degree polynomials, and {L}angevin dynamics.
\newblock {\em {SIAM} J. Comput.}, 53(1):1--46, 2024.

\bibitem{GamKizPerXu22}
D.~Gamarnik, E.~C. Kizildag, W.~Perkins, and C.~Xu.
\newblock Algorithms and barriers in the symmetric binary perceptron model.
\newblock In {\em 63rd {IEEE} Annual Symposium on Foundations of Computer
  Science, {FOCS} 2022, Denver, CO, USA, October 31 - November 3, 2022}, pages
  576--587. {IEEE}, 2022.

\bibitem{GamKW25}
D.~Gamarnik, E.~C. Kizildag, and L.~Warnke.
\newblock Optimal hardness of online algorithms for large independent sets.
\newblock 2025.
\newblock available online at \bl{\url{http://arxiv.org/abs/2504.11450}}.

\bibitem{GamarnikLi18}
D.~Gamarnik and Q.~Li.
\newblock Finding a large submatrix of a gaussian random matrix.
\newblock {\em The Annals of Statistics}, 46(6A):2511--2561, 2018.

\bibitem{GamMZ22}
D.~Gamarnik, C.~Moore, and L.~Zdeborova.
\newblock Disordered systems insights on computational hardness.
\newblock {\em Journal of Statistical Mechanics: Theory and Experiment},
  (11):115015, 2022.

\bibitem{GamarSud17}
D.~Gamarnik and M.~Sudan.
\newblock Limits of local algorithms over sparse random graphs.
\newblock {\em Ann. Probab.}, 45(4):2353--2376, 2017.

\bibitem{GamarSud17a}
D.~Gamarnik and M.~Sudan.
\newblock Performance of sequential local algorithms for the random {NAE-K-SAT}
  problem.
\newblock {\em SIAM Journal on Computing}, 46(2):590--619, 2017.

\bibitem{Gar88}
E.~Gardner.
\newblock The space of interactions in neural networks models.
\newblock {\em J. Phys. A: Math. Gen.}, 21:257--270, 1988.

\bibitem{GarDer88}
E.~Gardner and B.~Derrida.
\newblock Optimal storage properties of neural networks models.
\newblock {\em J. Phys. A: Math. Gen.}, 21:271--284, 1988.

\bibitem{GongHLS26}
S.~Gong, B.~Huang, S.~Li, and M.~Sellke.
\newblock Stable algorithms cannot reliably find isolated perceptron solutions.
\newblock 2026.
\newblock available online at \bl{\url{http://arxiv.org/abs/2604.00328}}.

\bibitem{Gordon85}
Y.~Gordon.
\newblock Some inequalities for {G}aussian processes and applications.
\newblock {\em Israel Journal of Mathematics}, 50(4):265--289, 1985.

\bibitem{HajekWX18}
B.~Hajek, Y.~Wu, and J.~Xu.
\newblock Submatrix localization via message passing.
\newblock {\em Journal of Machine Learning Research}, 18(186):1--52, 2018.

\bibitem{Huang24}
B.~Huang.
\newblock Capacity threshold for the {I}sing perceptron.
\newblock In {\em 65th {IEEE} Annual Symposium on Foundations of Computer
  Science, {FOCS} 2024, Chicago, IL, USA, October 27-30, 2024}, pages
  1126--1136. {IEEE}, 2024.

\bibitem{HuangSell24}
B.~Huang and M.~Sellke.
\newblock Optimization algorithms for multi-species spherical spin glasses.
\newblock {\em J. Stat. Phys.}, 191, 2024.

\bibitem{HuangS22a}
B.~Huang and M.~Sellke.
\newblock Tight {L}ipschitz hardness for optimizing mean field spin glasses.
\newblock {\em Comm. Pure. Appl. Math.}, 78(1):60--119, 2025.

\bibitem{KimRoc98}
J.~H. Kim and J.~R. Roche.
\newblock Covering cubes by random half cubes with applications to biniary
  neural networks.
\newblock {\em Journal of Computer and System Sciences}, 56:223--252, 1998.

\bibitem{KolarNRS11}
M.~Kolar, S.~Balakrishnan, A.~Rinaldo, and A.~Singh.
\newblock Minimax localization of structural information in large noisy
  matrices.
\newblock In {\em Advances in Neural Information Processing Systems 24 (NIPS
  2011)}, pages 909--917, 2011.

\bibitem{KraMez89}
W.~Krauth and M.~Mezard.
\newblock Storage capacity of memory networks with binary couplings.
\newblock {\em J. Phys. France}, 50:3057--3066, 1989.

\bibitem{LelargeMio18}
M.~Lelarge and L.~Miolane.
\newblock Fundamental limits of symmetric low-rank matrix estimation.
\newblock {\em Probability Theory and Related Fields}, 171(1-2):211--244, 2018.

\bibitem{Lesetal17}
T.~Lesieur, L.~Miolane, M.~Lelarge, F.~Krzakala, and L.~Zdeborová.
\newblock Statistical and computational phase transitions in spiked tensor
  estimation.
\newblock In {\em 2017 IEEE International Symposium on Information Theory
  (ISIT)}, pages 511--515, 2017.

\bibitem{LiSch24}
S.~Li and T.~Schramm.
\newblock Some easy optimization problems have the overlap-gap property.
\newblock In {\em The Thirty Eighth Annual Conference on Learning Theory, 30-4
  July 2025, Lyon, France}, volume 291 of {\em Proceedings of Machine Learning
  Research}, pages 3582--3622. {PMLR}, 2025.

\bibitem{MaWu15}
Z.~Ma and Y.~Wu.
\newblock Computational barriers in minimax submatrix detection.
\newblock {\em The Annals of Statistics}, 43(3):1089--1116, 2015.

\bibitem{MadOli04}
S.~C. Madeira and A.~L. Oliveira.
\newblock Biclustering algorithms for biological data analysis: a survey.
\newblock {\em IEEE/ACM Transactions on Computational Biology and
  Bioinformatics}, 1(1):24--45, 2004.

\bibitem{MMZ05}
M.~Mezard, T.~Mora, and R.~Zecchina.
\newblock Clustering of solutions in the random satisfiability problem.
\newblock {\em Physical Review Letters}, 94:197204, 2005.

\bibitem{MezardPZ02}
M.~Mezard, G.~Parisi, and R.~Zecchina.
\newblock Analytic and algorithmic solution of random satisfiability problems.
\newblock {\em Science}, 297(5582):812--815, 2002.

\bibitem{Mont15}
A.~Montanari.
\newblock Finding one community in a sparse graph.
\newblock {\em Journal of Statistical Physics}, 161:273--299, 2015.

\bibitem{Montanari19}
A.~Montanari.
\newblock Optimization of the {S}herrington-{K}irkpatrick hamiltonian.
\newblock In {\em 60th {IEEE} Annual Symposium on Foundations of Computer
  Science, {FOCS} 2019, Baltimore, Maryland, USA, November 9-12, 2019}, pages
  1417--1433. {IEEE} Computer Society, 2019.

\bibitem{MontRich14}
A.~Montanari and E.~Richard.
\newblock A statistical model for tensor {PCA}.
\newblock {\em Advances in Neural Information Processing Systems}, 27, 2014.

\bibitem{NakSun23}
S.~Nakajima and N.~Sun.
\newblock Sharp threshold sequence and universality for {I}sing perceptron
  models.
\newblock {\em {P}roceedings of the 2023 {A}nnual {ACM-SIAM} {S}ymposium on
  {D}iscrete {A}lgorithms ({SODA})}, pages 638--674, 2023.

\bibitem{OrenJBT26}
M.~Oren-Loberman, D.~Jerbi, T.~Bendory, and W.~Huleihel.
\newblock Inhomogeneous submatrix detection.
\newblock {\em arXiv preprint arXiv:2603.09602}, 2026.

\bibitem{PakKK23}
A.~Pak, J.~Ko, and F.~Krzakala.
\newblock Optimal algorithms for the inhomogeneous spiked {W}igner model.
\newblock In {\em Advances in Neural Information Processing Systems},
  volume~36, pages 5557--5586, 2023.

\bibitem{PerkXu21}
W.~Perkins and C.~Xu.
\newblock Frozen 1-{RSB} structure of the symmetric {I}sing perceptron.
\newblock {\em {STOC} 2021: Proceedings of the 53rd Annual {ACM SIGACT}
  {S}ymposium on {T}heory of {C}omputing}, pages 1579--1588, 2021.

\bibitem{PerryWB20}
A.~Perry, A.~S. Wein, and A.~S. Bandeira.
\newblock Statistical limits of spiked tensor models.
\newblock {\em Annales de l'Institut Henri Poincare, Probabilites et
  Statistiques}, 56(1):238--295, 2020.

\bibitem{PonGAR15}
B.~Pontes, R.~Giraldez, and J.~S. Aguilar-Ruiz.
\newblock Biclustering on expression data: A review.
\newblock {\em Journal of biomedical informatics}, 57:163--180, 2015.

\bibitem{HegKiz25}
A.~Hegade~K. R. and E.~C. Kizildag.
\newblock Large average subtensor problem: Ground-state, algorithms, and
  algorithmic barriers.
\newblock 2025.
\newblock available online at \bl{\url{http://arxiv.org/abs/2506.17118}}.

\bibitem{RahVir17}
M.~Rahman and B.~Virag.
\newblock Local algorithms for independent sets are half-optimal.
\newblock {\em The Annals of Probability}, 45(3):1543--1577, 2017.

\bibitem{ShabalinWPN09}
A.~A. Shabalin, V.~J. Weigman, C.~M. Perou, and A.~B. Nobel.
\newblock Finding large average submatrices in high dimensional data.
\newblock {\em The Annals of Applied Statistics}, 3(3):985--1012, 2009.

\bibitem{SchTir03}
M.~Shcherbina and B.~Tirozzi.
\newblock Rigorous solution of the {G}ardner problem.
\newblock {\em Comm. on Math. Physics}, (234):383--422, 2003.

\bibitem{Slep62}
D.~Slepian.
\newblock The one sided barier problem for {G}aussian noise.
\newblock {\em Bell System Tech. Journal}, 41:463--501, 1962.

\bibitem{SohnWein25}
Y.~Sohn and A.~S. Wein.
\newblock Sharp phase transitions in estimation with low-degree polynomials.
\newblock In {\em Proceedings of the 57th Annual ACM Symposium on Theory of
  Computing (STOC)}, pages 891--902, 2025.

\bibitem{Spen85}
J.~Spencer.
\newblock Six standard deviations suffice.
\newblock {\em Transactions of the American mathematical society},
  289(2):679--706, 1985.

\bibitem{StojnicCSetam09}
M.~Stojnic.
\newblock Various thresholds for $\ell_1$-optimization in compressed sensing.
\newblock 2009.
\newblock available online at \bl{\url{http://arxiv.org/abs/0907.3666}}.

\bibitem{StojnicICASSP10var}
M.~Stojnic.
\newblock $\ell_1$ optimization and its various thresholds in compressed
  sensing.
\newblock {\em ICASSP, IEEE International Conference on Acoustics, Signal and
  Speech Processing}, pages 3910--3913, 14-19 March 2010.
\newblock Dallas, TX.

\bibitem{StojnicISIT2010binary}
M.~Stojnic.
\newblock Recovery thresholds for $\ell_1$ optimization in binary compressed
  sensing.
\newblock {\em ISIT, IEEE International Symposium on Information Theory}, pages
  1593 -- 1597, 13-18 June 2010.
\newblock Austin, TX.

\bibitem{StojnicGardGen13}
M.~Stojnic.
\newblock Another look at the {G}ardner problem.
\newblock 2013.
\newblock available online at \bl{\url{http://arxiv.org/abs/1306.3979}}.

\bibitem{StojnicGardSphNeg13}
M.~Stojnic.
\newblock Negative spherical perceptron.
\newblock 2013.
\newblock available online at \bl{\url{http://arxiv.org/abs/1306.3980}}.

\bibitem{StojnicRegRndDlt10}
M.~Stojnic.
\newblock Regularly random duality.
\newblock 2013.
\newblock available online at \bl{\url{http://arxiv.org/abs/1303.7295}}.

\bibitem{Stojnicgscompyx16}
M.~Stojnic.
\newblock Fully bilinear generic and lifted random processes comparisons.
\newblock 2016.
\newblock available online at \bl{\url{http://arxiv.org/abs/1612.08516}}.

\bibitem{Stojnicgscomp16}
M.~Stojnic.
\newblock Generic and lifted probabilistic comparisons -- max replaces minmax.
\newblock 2016.
\newblock available online at \bl{\url{http://arxiv.org/abs/1612.08506}}.

\bibitem{Stojnicclupsk25}
M.~Stojnic.
\newblock A {CLuP} algorithm to practically achieve $\sim 0.76$ {SK}--model
  ground state free energy.
\newblock {\em Journal of Statistical Mechanics: Theory and Experiment},
  (11):123302, 2025.

\bibitem{Stojniccluphop25}
M.~Stojnic.
\newblock {CLuP} practically achieves $\sim 1.77$ positive and $\sim 0.33$
  negative {H}opfield model ground state free energy.
\newblock 2025.
\newblock available online at \bl{\url{http://arxiv.org/abs/2507.22396}}.

\bibitem{Stojnicalgbp25}
M.~Stojnic.
\newblock Binary perceptron computational gap -- a parametric fl-{RDT} view.
\newblock {\em Journal of Statistical Mechanics: Theory and Experiment},
  (4):043301, 2026.

\bibitem{Stojnictcmspnncaprdt23}
M.~Stojnic.
\newblock Capacity of the treelike sign perceptrons neural networks with one
  hidden layer -- rdt based upper bounds.
\newblock {\em IEEE Transactions on Information Theory}, 2026.
\newblock to appear (\bl{\url{http://arxiv.org/abs/2312.08244}}).

\bibitem{Stojnictcmspnncapliftedrdt23}
M.~Stojnic.
\newblock {\emph{Lifted}} {RDT} based capacity analysis of the 1-hidden layer
  treelike \emph{sign} perceptrons neural networks.
\newblock {\em Information and Inference: A Journal of the IMA}, 15(3):iaag034,
  09 2026.

\bibitem{Stojnicalgsbp26}
M.~Stojnic.
\newblock Parametric {RDT} approach to computational gap of symmetric binary
  perceptron.
\newblock {\em Journal of Statistical Physics}, 2026.
\newblock to appear. (\bl{\url{http://arxiv.org/abs/2601.10628}}).

\bibitem{StojnicNN27}
M.~Stojnic.
\newblock Exact capacity of the wide hidden layer treelike neural networks with
  generic activations.
\newblock {\em Neural Networks}, 205:109458, January 2027.

\bibitem{SunNobel13}
X.~Sun and A.~B. Nobel.
\newblock On the maximal size of large-average and anova-fit submatrices in a
  gaussian random matrix.
\newblock {\em Bernoulli}, 19(1):275--294, 2013.

\bibitem{Talbook11b}
M.~Talagrand.
\newblock {\em Mean field models and spin glasse: {V}olume {II}}.
\newblock A series of modern surveys in mathematics 55, Springer-Verlag, Berlin
  Heidelberg, 2011.

\bibitem{Talbook11a}
M.~Talagrand.
\newblock {\em Mean field models and spin glasses: {V}olume {I}}.
\newblock A series of modern surveys in mathematics 54, Springer-Verlag, Berlin
  Heidelberg, 2011.

\bibitem{Wein22}
A.~S Wein.
\newblock Optimal low-degree hardness of maximum independent set.
\newblock {\em Mathematical Statistics and Learning}, 4(3):221--251, 2022.

\bibitem{Wendel62}
J.~G. Wendel.
\newblock A problem in geometric probability.
\newblock {\em Mathematica Scandinavica}, 1:109--111, 1962.

\end{thebibliography}
\end{singlespace}

\end{document}